\documentclass[11pt, a4paper, twoside]{article}
\pdfoutput=1
\usepackage{lspaper}
\usepackage{genyoungtabtikz}
\usepackage{mfabacus}
\usepackage{tikzsymbols}
\usepackage{booktabs}
\usepackage{multirow, multicol}
\usepackage{placeins}
\usepackage[normalem]{ulem}
\usepackage{amssymb,amsmath}
\usepackage{pdflscape,afterpage}
\usetikzlibrary{matrix,arrows,decorations.pathmorphing,decorations.pathreplacing}
\usetikzlibrary{positioning,intersections,shapes.geometric,calc}
\usetikzlibrary{decorations.markings}
\RequirePackage{pxfonts}
\usepackage[toc,page]{appendix}
\usepackage{quiver}
\usepackage{scalefnt,caption}
\usepackage[all, knot]{xy}

\renewcommand\sss[1][n]{\mathfrak{S}_{#1}}
\renewcommand\spe[1]{\rspe{#1}}
\newcommand\rrspe[1]{\operatorname{S}_{#1}}
\renewcommand\D[1]{\rD{#1}}
\newcommand{\Ext}{\operatorname{Ext}}
\newcommand{\Rad}{\operatorname{Rad}}
\newcommand{\Soc}{\operatorname{Soc}}

\newcommand{\EXT}{\operatorname{EXT}}

\renewcommand\phi\varphi

\newcommand{\vn}{\varnothing}

\begin{document}

\title{Schurian-finiteness of blocks of type $A$ Hecke algebras}

\author{Susumu Ariki\\\normalsize Osaka University\\\normalsize Suita, Osaka, Japan 565-0871 \\\texttt{\normalsize ariki.susumu.ist@osaka-u.ac.jp}\\[11pt]
Sin\'ead Lyle
\\\normalsize University of East Anglia\\\normalsize Norwich Research Park, Norwich NR4 7TJ, UK\\
\texttt{\normalsize s.lyle@uea.ac.uk}\\[11pt]
Liron Speyer
\\\normalsize Okinawa Institute of Science and Technology\\\normalsize Onna-son, Okinawa, Japan 904-0495 \\\texttt{\normalsize liron.speyer@oist.jp}
}


\renewcommand\auth{Susumu Ariki, Sin\'ead Lyle, \& Liron Speyer}

\runninghead{Schurian-finiteness of blocks of type $A$ Hecke algebras}


\toptitle

\begin{abstract}
For any algebra $A$ over an algebraically closed field $\bbf$, we say that an $A$-module $M$ is Schurian if $\End_A(M) \cong \bbf$.
We say that $A$ is Schurian-finite if there are only finitely many isomorphism classes of Schurian $A$-modules, and Schurian-infinite otherwise.
By work of Demonet, Iyama and Jasso it is known that Schurian-finiteness is equivalent to $\tau$-tilting-finiteness, so that we may draw on a wealth of known results in the subject.
We prove that for the type $A$ Hecke algebras with quantum characteristic $e\geq 3$, all blocks of weight at least $2$ are Schurian-infinite in any characteristic.
Weight $0$ and $1$ blocks are known by results of Erdmann and Nakano to be representation finite, and are therefore Schurian-finite.
This means that blocks of type $A$ Hecke algebras (when $e\geq 3$) are Schurian-infinite if and only if they have wild representation type if and only if the module category has finitely many wide subcategories.
Along the way, we also prove a graded version of the Scopes equivalence, which is likely to be of independent interest.
\end{abstract}

\section{Introduction}

Let $A$ be a finite-dimensional algebra over an algebraically closed field $\bbf$.
All the $A$-modules we consider in this paper are assumed to be finite-dimensional left modules.
Then, an $A$-module $M$ is called Schurian if $\End_A(M) = \bbf$.
Schurian modules appear in various places.
Among them, the following two examples are well-known.

\begin{itemize}
\item[(1)]
Let $\theta:K_0(A\textup{-mod})\to \bbz$ be a linear map.
An $A$-module $M$ is called stable if $\theta(M)=0$ and any nonzero proper submodule $N$ satisfies $\theta(N)<0$.
Stable modules are Schurian.

\item[(2)]
Let $A$ be an $\bbf$-algebra.
If an $A$-module $M$ is an indecomposable module which belongs to the preprojective or the preinjective component of the Auslander--Reiten quiver of $A$, then $M$ is Schurian and $\Ext^1_A(M,M)=0$ -- for example, see~\cite[Chap.VIII, Lemma~2.7]{ASS1}.
\end{itemize}

An algebra $A$ is called Schurian-finite if the number of isomorphism classes of Schurian modules is finite.

\begin{rem}
In \cite{CKW}, they call an algebra $A$ Schur-representation-finite if there are finitely many isomorphism classes of Schurian $A$-modules of a fixed dimensional vector $d$, for each $d$.
\end{rem}

There had been few results on Schurian modules until recently, 
but the $\tau$-tilting theory initiated by Adachi, Iyama and Reiten \cite{AIR} has changed the perspective.
An $A$-module $M$ is $\tau$-rigid if $\Hom_A(M,\tau(M))=0$, where $\tau$ is the Auslander--Reiten translate.
An algebra $A$ is then called $\tau$-tilting finite if there are only finitely many isomorphism classes of indecomposable $\tau$-rigid $A$-modules.
Demonet, Iyama and Jasso \cite[Theorem~4.2]{DIJ} proved that the map that sends an $A$-module $M$ to $M/\Rad_B(M)$, where $B=\End_A(M)$, induces a bijection between isomorphism classes of indecomposable $\tau$-rigid $A$-modules and isomorphism classes of Schurian $A$-modules if $A$ is $\tau$-tilting finite, and that $\tau$-tilting finiteness coincides with Schurian finiteness.
Thus, Schurian-finiteness has reappeared under the name \emph{$\tau$-tilting finiteness}.

Let $M$ be a $\tau$-rigid $A$-module, and $P$ be a projective $A$-module.
We call the pair $(M,P)$
support $\tau$-tilting if $\Hom_A(P,M)=0$ and the numbers of pairwise non-isomorphic indecomposable direct summands of $M$ and $P$ sum up to the number of the isomorphism classes of simple $A$-modules.
$M$ is called a support $\tau$-tilting module.
Then, basic support $\tau$-tilting modules are in bijection with functorially-finite torsion classes in $A\textup{-mod}$, and with two-term silting complexes in $K^b(A\textup{-proj})$~\cite[Theorem~0.5]{AIR}.
In other words, the study of Schurian modules has applications to those representation theoretic classifications.
Another application is that if $A$ is wild and Schurian-finite, then it gives an example of a wild algebra which is not strictly wild.
Recently, the relationship with polytope theory has also been pursued.

Because of its importance, researchers who study $\tau$-tilting theory look at various examples to check whether they are $\tau$-tilting finite or not.
For example, we know that preprojective algebras of Dynkin type are Schurian-finite by~\cite{Mizuno14}.
Note that if $A$ is representation-finite, then $A$ is Schurian-finite.
An immediate consequence of this is that blocks of Hecke algebras of weight $0$ or $1$ are Schurian-finite, by \cite{enreptype,APA1}.
The converse may not hold.
Indeed, preprojective algebras of type other than $A_n$ for $1\leq n \leq 4$ are representation-infinite.
Another example is a multiplicity-free Brauer cyclic graph algebra with an odd number of vertices~\cite{Ad2}.
More recently still, it was shown by Miyamoto and Wang~\cite[Corollary~3.14]{mw22} that if the characteristic of the base field is not $2$, then any self-injective cellular algebra of polynomial growth is Schurian-finite.

However, it is observed that the converse \emph{does} hold for many classes of algebras.
The first obvious examples are those which admit preprojective component,
such as the class of path algebras of acyclic quivers, quasi-tilted algebras~\cite{ch97}, and algebras that satisfy the separation condition~\cite[Chapter IX, Theorem 4.5]{ASS1}.
Other examples include cycle-finite algebras~\cite{MS16}, gentle algebras~\cite{Plam19}, tilted and cluster-tilted algebras~\cite{zito20},
locally hereditary algebras~\cite{AHMW}, and simply-connected algebras~\cite{qiwangsimplyconnected}.
As algebras satisfying the separation condition are simply-connected~\cite{skow93}, it also implies that representation finiteness coincides with Schurian-finiteness for them.

In this article, we initiate the study of Schurian-finiteness for block algebras of the Hecke algebra $\hhh$ of the symmetric group and add a new family of algebras for which Schurian-finiteness coincides with representation finiteness.
The difficulties to overcome were twofold.
Recall that block algebras of the Hecke algebra of the symmetric group $\sss$ are labeled by $e$-cores $\rho$ such that $n-|\rho|\in e\bbz_{\geq0}$, where $e$ is the \emph{quantum characteristic}.
The nonnegative integer $(n-|\rho|)/e$ is called the weight of the block algebra.
As all of the known methods to determine the Schurian-infiniteness of a finite-dimensional algebra are based on the bound quiver presentation of the algebra, we have to obtain information of the Gabriel quiver of each block algebra of $\hhh$ which suffices to determine the Schurian-infiniteness.
Then, the first difficulty was the lack of a method of reduction to small $n$ (except for the Scopes equivalence), that would allow us to determine the Schurian-infiniteness by explicit computation.
For example, the induction functor does not behave well with the Schurian property.
On the other hand, we may use the induction functor to determine the representation type of block algebras of $\hhh$.
Thus, we must control the Gabriel quiver of the block algebra, for large $n$, itself.
Then, the second difficulty was the fact that it is not easy to compute the Gabriel quiver of the block even if one tries to find a small piece of information on the structure of indecomposable projective modules.
Indeed, when we study algebras outside of bound quiver algebras, this is the most difficult part.
Curiously, this reality is often overlooked.

The key idea to overcome those difficulties is \cref{prop:matrixtrick}.
It asserts that if we find a certain set of partitions for which the graded decomposition numbers labelled by them satisfy certain conditions, then we may find enough information to determine Schurian-infiniteness.
Because of the proposition, we are able to focus on graded decomposition numbers exclusively to prove our main theorems.
This assertion may sound a bit surprising because graded decomposition numbers of a graded cellular algebra are determined by the Grothendieck group of its graded modules and they do not contain information on the Gabriel quiver in general.
When we consider the blocks of $\hhh$ in positive characteristic that appear in \cref{prop:matrixtrick}, the submatrix of graded decomposition numbers that satisfy the necessary extra conditions still does not give us information on extensions between the corresponding simple modules for $\hhh$.
Rather, we deduce extensions between simples for some idempotent truncation, from which we can deduce Schurian-infiniteness of the truncated algebra, and lift the property back to $\hhh$.
To apply \cref{prop:matrixtrick}, we must find partitions that satisfy the assumptions of the proposition.
We note that this requires insight, and another new idea here is in the novel use of various runner-removal and row-removal theorems, which allow us to perform the necessary computations in order to find a submatrix of decomposition numbers among the list in \cref{prop:matrixtrick}.

In small characteristics, we often see different behaviour than the general patterns.
Thus, we need a different method to determine the Schurian-finiteness.
In particular, we assume that the quantum characteristic is $e\ne 2$ for our main theorem.

\begin{rmk}
If $e=2$ and $p\neq 2$, then blocks of weight $0$, $1$ or $2$ are Schurian-finite.
By Erdmann and Nakano~\cite{enreptype}, these weight $2$ blocks have infinite tame representation type.
When $e=2$, blocks of weight at least $3$ have wild representation type, but our methods cannot determine whether these blocks are Schurian-infinite.

Morita classes of tame blocks are classified by the first author~\cite{ariki21} under the assumption $p\ne2$.
Then Wang showed in~\cite[Theorem~5.5]{qiwang2point} that all those tame blocks of Hecke algebras are $\tau$-tilting finite (and therefore Schurian-finite).
In~\cite[Theorem~3.4]{qiwang22}, Wang also showed that blocks of tame Schur algebras $S(n, r)$, which occur if $p=2$, are $\tau$-tilting finite, from which it follows that the (indecomposable) Hecke algebra $\hhh[4]$ is also Schurian-finite if $e=p=2$.
This is the only example of a block of $\hhh$ of weight at least $2$, for $e=p=2$, for which we can determine whether it is Schurian-finite.
\end{rmk}

The following theorem is the main result of this article.

\begin{thm}\label{thm:main}
Suppose $e\geq3$ and that $B$ is a block of $\hhh$ of weight at least $2$.
Then $B$ is Schurian-infinite in any characteristic.
In particular, if $e\geq 3$, then $B$ is Schurian-finite if and only if it is representation-finite.
\end{thm}

\begin{rem}
In an earlier version of this paper, we needed a result by Bowman and the third author~\cite{bs15}, Fock space theory, cf.~\cite{arikibook}, and a result by the first author and Mathas~\cite[Corollary 3.7]{am04}, to settle several cases in weight four.
In the current version we have developed a more combinatorial method to replace the canonical basis computation, and have merged the original paper by the first and third author with its sequel by the second and third.
\end{rem}

Recall that a full subcategory of an abelian category is a wide subcategory if it is closed under isomorphisms, extensions, kernels and cokernels.
By \cite[Proposition 2.24]{AsaiSemibricks}, we have a bijection between wide subcategories and isomorphism classes of semibricks.
A semibrick is a set of pairwise Hom-orthogonal Schurian modules, and we say that a semibrick is left-finite if the smallest torsion class that contains the semibrick is functorially-finite.
Every semibrick is left-finite if the algebra is Schurian-finite by \cite[Theorem 3.8]{DIJ}.
Taking the bijection between isomorphism classes of left-finite semibricks and isomorphism classes of basic support $\tau$-tilting modules into consideration, we have the following corollary.
 
\begin{cor}
Suppose $e\geq 3$ and that $B$ is a block of $\hhh$.
Then the following are equivalent.
\begin{enumerate}[label=(\roman*)]
\item $B$ is representation-finite.
\item $B$ is Schurian-finite.
\item The category of finite-dimensional $B$-modules has finitely many wide subcategories.
\end{enumerate}
\end{cor}

The paper is organised as follows.
In \cref{sec:background}, we recall various runner-, row-, and column-removal theorems in \cref{subsec:grdec} after we fix notations for Hecke algebras, partitions and abacus displays.
Our method to show Schurian-infiniteness by using graded decomposition numbers is explained in \cref{subsec:reduction}.
We also prove a graded analogue of the Scopes equivalence in \cref{sec:gradedScopes}, which will be of independent interest.
After \cref{sec:gradedScopes}, we exclusively work with graded decomposition numbers -- except for in a few cases -- to 
find partitions that satisfy the assumptions of \cref{prop:matrixtrick}.
As much is known for blocks of weight $2$ and $3$, we give a classification of Schurian-finiteness/infiniteness for those blocks in \cref{sec:wt2,sec:wt3} whenever $e\geq 3$, proving \cref{thm:main} in those cases.
Our approach there largely hinges on the runner-removal result of James and Mathas.
In \cref{sec:highwt}, we assume that $e\geq3$ and focus on blocks of weight greater than or equal to $4$.

\begin{ack}
The first author is partially supported by JSPS Kakenhi grants number 18K03212 and 21K03163.
The second author was supported by an LMS Emmy Noether Fellowship grant which enabled her to visit OIST.
The third author is partially supported by JSPS Kakenhi grant number 20K22316.
The third author thanks Osaka University for their hospitality during his visit, in which this research was initiated.
The first author is grateful for a pleasant stay at OIST in November 2021.
During the visit, we completed most of the work on this paper.
We thank Matt Fayers for making his LaTeX style file for abacus displays publicly available, for pointing out a useful reference within \cite{richardswt2} that we use to characterise \emph{adjacent} partitions in \cref{sec:wt2}, and for helpful conversations when we applied his results to a concrete example.
We also thank Andrew Mathas for providing us with the graded decomposition matrices for various blocks of low rank when $e=3$ and $p>0$ -- these were a great help in proving our result for those blocks.
The third author thanks Eoghan McDowell for help with TikZ.

The first author considered the Schurian-finiteness of block algebras of Hecke algebras two years ago with Ryoichi Kase, Kengo Miyamoto, Euiyong Park, and Qi Wang, trying to find an explicit subquotient algebra and prove its Schurian-infiniteness.
Although unsuccessful in this endeavour, the first author thanks those four colleagues for their time spent together on the project in 2019.
\end{ack}

\section{Background}\label{sec:background}

Throughout, we let $\bbf$ denote an algebraically closed field of characteristic $p\geq 0$.
All our modules are left modules.

\subsection{Partitions}\label{subsec:partitions}

A partition of $n$ is a weakly decreasing sequence of nonnegative integers $\la = (\la_1, \la_2, \dots)$ that sum to $n$.
We write $\la \vdash n$ to mean `$\la$ is a partition of $n$'.
It will be useful to consider each partition to have infinite length, though it will only have finitely many nonzero terms.
Thus we will omit the trailing zeroes when writing $\la$, and always consider $\la_r = 0$ for $r$ large enough.
We will also group together equal terms, so that, for example, we will write $(3,1^2)$ instead of $(3,1,1)$.
We denote the unique partition of $0$ by $\varnothing$.

For $e \in \bbz_{\geq 2}$, we say that a partition $\la$ is \emph{$e$-regular} if it does not have any $e$ nonzero parts being equal, and \emph{$e$-singular} if it does.
For example, $(5,1^2)$ is $3$-regular, but $(4,1^3)$ and $(3,1^4)$ are both $3$-singular.

The \emph{conjugate partition} $\la'$ is defined by 
\[
\la'_i = |\{ j\geq 1 \mid \la_j \geq i\}|.
\]

For a partition $\la$, its \emph{Young diagram} is the set $[\la] = \{(r,c) \in \bbn \times \bbn \mid c \leq \la_r\}$, which we may depict as boxes in the plane, following the English convention, as in the example below.
To each node $A = (r,c) \in [\la]$, we assign the \emph{$e$-residue} $i \in \{0, 1, \dots, e-1\}$ with $i \equiv c-r \pmod e$.

If a node $A \in [\la]$ (respectively $B \notin [\la]$) is such that $[\la]\setminus A$ (respectively $[\la] \cup B$) is a Young diagram for a partition, then we say that $A$ is a \emph{removable node} (respectively $B$ is an \emph{addable node}).
We refer to an addable or removable node $(r,c)$ as an \emph{addable $i$-node} or a \emph{removable $i$-node}, respectively, if $i \equiv c-r \pmod e$.

\begin{eg}
Let $e=4$, and take $\la = (6,4,3,1^2)$.
Then $[\la]$ is drawn below, with the residues written in each node, as well as the residues of addable nodes.
\[
\Yaddables1
\yngres(4,6,4,3,1^2)
\]
\end{eg}

If $\la$ and $\mu$ are partitions of $n$, we say that $\la$ dominates $\mu$, and write $\la \dom \mu$, if $\la_1 + \la_2 + \dots + \la_r \geq \mu_1 + \mu_2 + \dots + \mu_r$ for all $r$, and write $\la \doms \mu$ if $\la \dom \mu$ and $\la \neq \mu$.

\subsection{Hecke algebras}\label{subsec:hecke}

Let $q\in \bbf^\times$.
The Hecke algebra of the symmetric group, denoted by $\hhh$, is the unital associative $\bbf$-algebra generated by $T_1, T_2, \dots, T_{n-1}$ subject to the relations
\begin{alignat*}3
(T_i - q)(T_i + 1) &= 0 \qquad &&\text{ for } i = 1,\dots, n-1;\\
T_i T_j &= T_j T_i \qquad &&\text{ for } 1\leq i,j \leq n-1 \text{ with } |i-j|>1;\\
T_i T_{i+1} T_i &= T_{i+1} T_i T_{i+1} \qquad &&\text{ for } i = 1,\dots, n-2.
\end{alignat*}

An excellent introduction to the representation theory of $\hhh$, which largely parallels that of the symmetric group $\sss$, can be found in \cite{mathas}.
Here, we will briefly recall some key aspects that we will require.

The \emph{quantum characteristic} of $\hhh$ is the smallest positive integer $e$ such that $1 + q + q^2 + \dots + q^{e-1} = 0$, if such an $e$ exists, and we set $e = \infty$ otherwise.

The Hecke algebras are cellular algebras.
For each partition $\la$ of $n$, one may construct the Specht module $\spe\la$, which is also a cell module, with cellular basis given by the \emph{dual Murphy basis} -- see, for example, \cite{durui01,hm10}.
Note that our Specht modules agree with those of Dipper and James~\cite{dj86}.
If $e > n$, $\hhh$ is semisimple, and the set $\{\spe\la \mid \la \vdash n\}$ is a complete set of pairwise non-isomorphic irreducible $\hhh$-modules.
If $e \leq n$, then $\spe\la$ has a simple head, denoted $\D\la$, whenever $\la$ is $e$-regular.
The set $\{\D\la \mid \la \vdash n, \text{ $\la$ is $e$-regular}\}$ is a complete set of pairwise non-isomorphic irreducible $\hhh$-modules when $e < \infty$.

The Specht modules $\spe\la$ can be constructed very explicitly, and have a basis indexed by \emph{standard $\la$-tableaux}.
However, the simple modules $\D\la$ are much harder to explicitly construct, except in special cases.
In general, even the dimensions of $\D\la$ are unknown.

It is well-known that the Specht modules are indecomposable if $e\neq 2$ -- for example, by \cite[Corollary 8.6]{dj91}.
We will abuse notation a little and say that two partitions $\la$ and $\mu$ lie in the same block of $\hhh$ whenever $\spe\la$ and $\spe\mu$ do.
If we remove all length $e$ rim hooks from a Young diagram $[\la]$, then we obtain a Young diagram for its \emph{$e$-core}, i.e.~a partition whose Young diagram has no removable length $e$ rim hooks.
The following result is well-known, and is often referred to as Nakayama's conjecture.

\begin{thmc}{mathas}{Corollary 5.38}
Two partitions of $\hhh$ are in the same block if and only if they have the same $e$-core.
\end{thmc}

Armed with this, we may index the blocks of Hecke algebras by their cores and weights, and will denote by $B(\rho,w)$ the block of $\hhh[|\rho|+ew]$ with core $\rho$ and weight $w$.
More recently, Khovanov and Lauda~\cite{kl09}, and, independently, Rouquier~\cite{rouq}, have introduced \emph{quiver Hecke algebras}, or \emph{KLR algebras}, in order to categorify the negative halves of quantum groups.
Khovanov and Lauda also introduced cyclotomic quotients in their paper, that we refer to as \emph{cyclotomic KLR algebras}, that categorify highest weight irreducible modules over quantum groups~\cite{kk12}.

Importantly for us, Brundan and Kleshchev showed in \cite{bkisom} that if $e = \infty$, $\hhh$ is isomorphic to a level 1 cyclotomic KLR algebra of type $A_\infty$, while if $e<\infty$, $\hhh$ is isomorphic to a level 1 cyclotomic KLR algebra of type $A^{(1)}_{e-1}$.

We will not recall the (long!)~presentation of the cyclotomic KLR algebras, as we will not be directly working with the definition.
For our purposes, it will suffice to note that this framework allows us to study the graded representation theory of $\hhh$, which is further developed in \cite{bk09,bkw11}.
In \cref{subsec:grdec} we will discuss this further.

\subsection{Abacus combinatorics}\label{subsec:abacus}

For a partition $\la$ of $n$, we define its \emph{beta-numbers} to be $\beta_i = \la_i - i$, for all $i\geq 1$.
The $e$-runner abacus is drawn with $e$ infinite vertical runners, with marked positions increasing from left-to-right along successive `rows'.
Our convention is that the 0 position is on the leftmost runner.
For example, if $e=4$, our abacus is marked as follows.

{\scriptsize
\[
\begin{tikzpicture}[scale=.6]
\foreach\x in{0,...,3}{\draw(\x,-2.5)--++(0,5);\draw[dashed](\x,-2.5)--++(0,-1);\draw[dashed](\x,2.5)--++(0,1);}
\foreach\x in{0,...,3}{\draw(\x,0)node[fill=white]{$\x$};};
\foreach\x in{-8,...,-5}\draw(\x+8,2)node[fill=white]{$\x$};
\foreach\x in{-4,...,-1}\draw(\x+4,1)node[fill=white]{$\x$};
\foreach\x in{4,...,7}\draw(\x-4,-1)node[fill=white]{$\x$};
\foreach\x in{8,...,11}\draw(\x-8,-2)node[fill=white]{$\x$};
\end{tikzpicture}
\]
}

The \emph{abacus display} for $\la$ is then obtained by placing a bead in position $\beta_i$, for each $i \geq 1$.
The $e$-core of $\la$ is obtained by stripping off all possible $e$-rim hooks from the Young diagram of $\la$, which translates to pushing all beads up as high as possible on the abacus.

We will adopt the notation and conventions of \cite{fayerswt2data}, which is in turn a sleeker presentation of the results of \cite{richardswt2}.
If $B$ is a block with core $\rho$, we take the abacus display for $\rho$ and define the integers $p_0 < p_1 < \dots < p_{e-1}$ so that each is the position of the lowest bead on one of the runners.
If $B$ is a weight 2 block, then the abacus display for any partition in $B$ is obtained from that for its $e$-core by sliding two beads down one place, or by sliding one bead down two places.

Then, for $0 \leq i \leq j < e$, define
\[
_iB_j = \begin{cases}
1 &\text{if } p_j - p_i < e,\\
0 &\text{if } p_j - p_i > e.
\end{cases}
\]
Collectively, the array $_iB_j$ is the \emph{pyramid} of the weight 2 block corresponding to the $e$-core we started with.
We extend this to all pairs of integers by setting $_iB_j = 0$ if $i<0$ or $j\geq e$, and $_iB_j = 1$ if $i>j$.
Finally, we adopt the shorthand notation $_i0_j$ and $_i1_j$ to mean that $_iB_j = 0$ and $1$, respectively.
In particular, we note that $_i0_e$ for any $i$.

\begin{eg}
Let $e=4$, and let $\rho$ be the 4-core $(2^2)$.
Then the corresponding beta-numbers are $\beta = (1,0,-3,-4,\dots)$, and the corresponding abacus display is below.
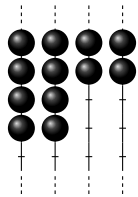
\[
\abacus(vvvv,bbbb,bbbb,bbnn,bbnn,nnnn,vvvv)
\]
Note that we have $p_0 = -6$, $p_1 = -5$, $p_2 = 0$, $p_3 = 1$, and thus that $_01_1$, $_00_2$, $_00_3$, $_10_2$, $_10_3$, $_21_3$.
We always have that $_i1_i$, since $p_i - p_i = 0 < e$.
\end{eg}

Next, we introduce notation for the partitions in a weight 2 block.
Number the runners from $0$ to $e-1$ so that runner $i$ contains the marked position $p_i$, recalling that $p_0 < p_1 < \dots < p_{e-1}$.
If the abacus display of $\la$ is obtained from that of its $e$-core by sliding the lowest beads on runners $i$ and $j$ each down one spot, with $i<j$, we denote the partition by $\langle i, j \rangle$.
If it is obtained by sliding the lowest bead on runner $i$ down two spaces, we denote it by $\langle i \rangle$.
Finally, if it is obtained by sliding each of the bottom two beads on runner $i$ down one space, we denote it $\langle i^2 \rangle$.

Thus we may denote any partition in $B$ by one of the above, for some $i$ (and possibly $j$).
These partitions thus index all Specht modules in the weight two block $B$.

\begin{eg}
Continuing our previous example, where $e=4$, and we consider a block $B$ with 4-core $\rho = (2^2)$, and weight $2$.
Then the abacus displays for $\langle 3 \rangle$, $\langle 2 \rangle$, $\langle 1 \rangle$, and $\langle 0 \rangle$ are below, yielding partitions $(10,2)$, $(9,3)$, $(4,3^2,1^2)$, $(3^3,1^3)$, respectively.
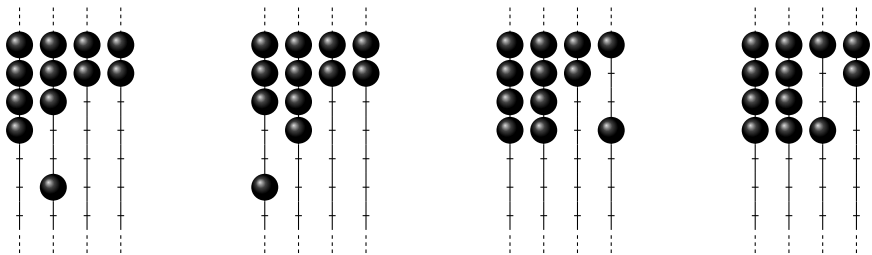
\[
\abacus(vvvv,bbbb,bbbb,bbnn,bnnn,nnnn,nbnn,nnnn,vvvv) \qquad \qquad \abacus(vvvv,bbbb,bbbb,bbnn,nbnn,nnnn,bnnn,nnnn,vvvv) \qquad \qquad \abacus(vvvv,bbbb,bbbn,bbnn,bbnb,nnnn,nnnn,nnnn,vvvv) \qquad \qquad \abacus(vvvv,bbbb,bbnb,bbnn,bbbn,nnnn,nnnn,nnnn,vvvv)
\]
The abacus displays for $\langle 2, 3 \rangle$, $\langle 1,3 \rangle$, $\langle 3^2 \rangle$, and $\langle 2^2 \rangle$ are below, yielding partitions $(6^2)$, $(6,2^2,1^2)$, $(6,3^2)$, $(5,3^2,1)$, respectively.
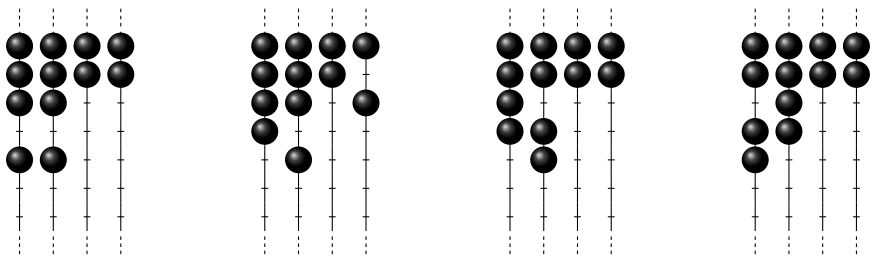
\[
\abacus(vvvv,bbbb,bbbb,bbnn,nnnn,bbnn,nnnn,nnnn,vvvv) \qquad \qquad \abacus(vvvv,bbbb,bbbn,bbnb,bnnn,nbnn,nnnn,nnnn,vvvv) \qquad \qquad \abacus(vvvv,bbbb,bbbb,bnnn,bbnn,nbnn,nnnn,nnnn,vvvv) \qquad \qquad \abacus(vvvv,bbbb,bbbb,nbnn,bbnn,bnnn,nnnn,nnnn,vvvv)
\]
\end{eg}

\subsection{Graded decomposition matrices}\label{subsec:grdec}

Let $\la, \mu \vdash n$, with $\mu$ $e$-regular.
The corresponding \emph{decomposition number} is the multiplicity $d_{\la\mu}^{e,p}(1) = [\spe\la : \D\mu]$ of $\D\mu$ in $\spe\la$.
For a graded $\hhh$-module $D$, let $D\langle d \rangle$ denote the graded shift (by $d$) of the module $D$ -- in other words $D\langle d \rangle_r = D_{r-d}$.
Then the corresponding \emph{graded} decomposition number is the Laurent polynomial
\[
d_{\la\mu}^{e,p}(v) = [\spe\la : \D\mu]_v = \sum_{d\in\bbz} [\spe\la : \D\mu \langle d \rangle] v^d \in \bbn[v,v^{-1}].
\]
It is known that $d_{\la\la}^{e,p}(v) = 1$ and $d_{\la\mu}^{e,p}(v) \neq 0$ only if $\la \domby \mu$ -- this follows, for instance, from the graded cellular basis of Hu--Mathas~\cite[Lemma~2.13, Theorem~6.11 and Section~6.4]{hm10}.

It is also well-known that the \emph{ungraded} decomposition matrices $D_p = (d_{\la\mu}^{e,p}(1))_{\la,\mu}$ and $D_0 = (d_{\la\mu}^{e,0}(1))_{\la,\mu}$ are related by multiplication by the so-called \emph{adjustment matrix}.
For us, we will need the more recent development of the \emph{graded adjustment matrix}, introduced in~\cite[Section~5.6]{bk09}.

By~\cite[Theorem~5.17]{bk09}, $D_p$, the characteristic $p$ graded decomposition matrix, is obtained from the characteristic 0 one $D_0$ by post-multiplication by the lower-unitriangular adjustment matrix, whose entries $a_{\la\mu}(v)$ -- indexed by $\la$ and $\mu$ both $e$-regular partitions of $n$ -- are Laurent polynomials with nonnegative integral coefficients, symmetric in $v, v^{-1}$.
In other words,
\[
d_{\la\mu}^{e,p}(v) = d_{\la\mu}^{e,0}(v) + \sum_{\nu \domsby \mu} d_{\la\nu}^{e,0}(v) a_{\nu\mu}(v).
\]

\begin{lem}\label{lem:charfree}
Suppose that $\la$ and $\mu$ are partitions, with $\mu$ $e$-regular, and that $d_{\la\mu}^{e,p}(1) = d_{\la\mu}^{e,0}(1)$ for a prime $p$.
Then $d_{\la\mu}^{e,p}(v) = d_{\la\mu}^{e,0}(v)$.
\end{lem}

\begin{proof}
Recall that
\[
d_{\la\mu}^{e,p}(v) = d_{\la\mu}^{e,0}(v) + \sum_{\nu \domsby \mu} d_{\la\nu}^{e,0}(v) a_{\nu\mu}(v).
\]
Since $d_{\la\mu}^{e,p}(1) = d_{\la\mu}^{e,0}(1)$, we have that $\sum_{\nu \domsby \mu} d_{\la\nu}^{e,0}(v) a_{\nu\mu}(v) = 0$, and the result follows.
\end{proof}

Our next result will ensure that our calculations for weight 2 and 3 blocks may be characteristic-free, so long as the characteristic is larger than the weight of the block.
For $w=2$, $3$, or $4$, the following is proved in \cite{richardswt2}, \cite{fay08wt3} and \cite{fay07wt4}, respectively.

\begin{thm}\label{thm:jamesconj}
\label{thm:jamesconjsmallwt}
If $B$ is a block of weight $w \leq 4$, then the adjustment matrix is the identity matrix whenever $p>w$.
\end{thm}

Note that Low has also extended the above results to $q$-Schur algebras~\cite{lowwt34}.

In weight 2, the characteristic $2$ situation has been solved by Fayers; we will use the following two results in \cref{sec:wt2}, in which we employ the pyramid notation introduced in \cref{subsec:abacus}.

\begin{thmc}{faywt2}{Corollary 2.4}\label{thm:wt2adjust}
Let $\nu$ and $\mu$ be $e$-regular partitions in a weight $2$ block of $\hhh$, and let $p=2$.
Then
\[
a_{\nu\mu}(v) = 
\begin{cases}
1 &\text{ if } \nu = \langle i^2 \rangle, \ \mu =\langle i\rangle, \ _{i-1}0_i \text{ and } _i0_{i+1} \text{ for } 1\leq i \leq e-1;\\
1 &\text{ if } \nu = \langle i^2 \rangle, \ \mu =\langle i, i+1\rangle, \ _{i-1}0_i \text{ and } _i1_{i+1} \text{ for } 1\leq i < e-1;\\
\delta_{\la\mu} &\text{ otherwise.}\\
\end{cases}
\]
\end{thmc}

Recall that the Mullineux map $m$ is a bijection on the set of $e$-regular partitions satisfying
$\D\la \otimes \sgn \cong \D{m(\la)}$ if $e=p$ (so that $\hhh \cong \bbf\sss$), and an analogous statement for the Hecke algebra, corresponding to twisting simple modules by a certain sign automorphism of $\hhh$.
There are several combinatorial algorithms to compute $m$, due to several different authors, but we will not need these here.

The following result will be useful for us later.

\begin{propc}{fay09}{Lemma 3.6 and Proposition 3.7}\label{prop:degrees}
Let $\la$ and $\mu$ be partitions of $n$, and suppose that $\mu$ is $e$-regular.
Then $d_{\la\mu}^{e,0}(v) = 0$ unless $\mu \dom \la \dom m(\mu)'$, where $m$ is the Mullineux map.
Moreover, if $\la = m(\mu)'$, then $d_{\la\mu}^{e,0}(v) = v^w$, where $w$ is the weight of the block containing $\la$ and $\mu$, and $d_{\la\mu}^{e,0}(v)$ has degree at most $w-1$ otherwise.
\end{propc}

In \cite{faywt2}, a partition $\la$ is said to be \emph{adjacent} to an $e$-regular partition $\mu$ if $\mu \dom \la \dom m(\mu)'$ and $|\partial\la - \partial\mu| = 1$, where $\partial$ is defined in \cite[Section 1.4]{faywt2}.
We will not require this definition, and since we only consider situations where both $\la$ and $\mu$ are $e$-regular we will say that $\la$ and $\mu$ are adjacent, without worrying about which is more dominant.
Just above \cite[Theorem~4.4]{richardswt2}, Richards defines a map $\diamond$ on $e$-regular weight two partitions that satisfies $\partial \mu^{\diamond'} = \partial \mu$, and on p399 of loc.~cit., Richards shows that this map agrees with the Mullineux map.
Then by \cite[Theorem~4.4]{richardswt2}, combined with \cref{prop:degrees}, $\la$ and $\mu$ are adjacent if and only if $d_{\la\mu}^{e,0}(v) = v$, where $\mu \doms \la  \doms m(\mu)'$.

\begin{thmc}{faywt2}{Theorem 3.2}\label{thm:wt2ext}
Suppose $p=2$, $B$ is a weight two block and $\la$, $\mu$ are $e$-regular partitions in $B$.
Then:
\begin{itemize}
\item If neither of $\la$ and $\mu$ is of the form $\langle i^2 \rangle$ for some $i$ such that $_{i-1}0_i$, then $\Ext^1(\D\la, \D\mu) \cong \bbf$ if $\la$ and $\mu$ are adjacent, and is trivial otherwise.

\item Suppose $\la = \langle i^2 \rangle$ for some $i$ such that $_{i-1}0_i$.
Then $\Ext^1(\D\la, \D\mu) \cong \bbf$ if $\mu = \langle i \rangle$ with $_i0_{i+1}$ or if $\mu = \langle i, i+1 \rangle$ with $_i1_{i+1}$, and is trivial otherwise.
\end{itemize}
\end{thmc}

\subsection{Runner removal}\label{subsec:runnerem}

Our next result is a runner removal theorem of James and Mathas, though we use a useful reformulation of the statement due to Fayers~\cite[Theorem~2.15]{fay07wt4}.

\begin{thmc}{jm02}{Theorem 3.2}\label{thm:runnerrem}
Suppose $e\geq 3$ and that $\la$ and $\mu$ are partitions of $n$, $\mu$ is $e$-regular, and that we take abacus displays for $\la$ and $\mu$.
Suppose that for some $i$, the last bead on runner $i$ occurs before every unoccupied space on both abacus displays, and define two abacus displays with $e-1$ runners by deleting runner $i$ from the abacus displays of $\la$ and $\mu$.
Let $\la^-$ and $\mu^-$ be the partitions defined by these displays.
If $\mu^-$ is $(e-1)$-regular, then
\[
d^{e,0}_{\la\mu}(v) = d^{e-1,0}_{\la^- \mu^-}(v).
\]
\end{thmc}

\subsection{Row-removal}\label{subsec:rowrem}

The following so-called row- and column-removal theorems will be used in \cref{sec:highwt} for determining that blocks of weight at least 4 are Schurian-infinite.

\begin{thmc}{cmt02}{Theorem 1}\label{thm:rowremFock}
Let $\la = (\la_1, \la_2, \dots)$ and $\mu = (\mu_1, \mu_2, \dots)$.
\begin{enumerate}[label=(\roman*)]
\item If $\la_1 + \la_2 + \dots + \la_r = \mu_1 + \mu_2 + \dots + \mu_r$ for some $r$, and we let

\begin{alignat*}3
\la^{(0)} &= (\la_1, \la_2, \dots, \la_r ), \qquad &&\mu^{(0)} = (\mu_1, \mu_2, \dots, \mu_r),\\
\la^{(1)} &= (\la_{r+1}, \la_{r+2}, \dots ), \qquad &&\mu^{(1)} = (\mu_{r+1}, \mu_{r+2}, \dots),
\end{alignat*}
then $d_{\la\mu}^{e,0}(v) = d_{\la^{(0)}\mu^{(0)}}^{e,0}(v) d_{\la^{(1)}\mu^{(1)}}^{e,0}(v)$.

\item If $\la_1' + \la_2' + \dots + \la_r' = \mu_1' + \mu_2' + \dots + \mu_r'$ for some $r$, and we let

\begin{alignat*}3
\la^{(0)} &= (\min(\la_1,r), \min(\la_2,r), \dots), \qquad &&\mu^{(0)} = (\min(\mu_1,r), \min(\mu_2,r), \dots),\\
\la^{(1)} &= (\max(\la_1-r,0), \max(\la_2-r,0), \dots ), \qquad &&\mu^{(1)} = (\max(\mu_1-r,0), \max(\mu_2-r,0), \dots),
\end{alignat*}
then $d_{\la\mu}^{e,0}(v) = d_{\la^{(0)}\mu^{(0)}}^{e,0}(v) d_{\la^{(1)}\mu^{(1)}}^{e,0}(v)$.
\end{enumerate}
\end{thmc}

\begin{thmc}{Donkin}{4.2(9) and 4.2(15)}\label{thm:rowremDecomp}
If $\la$, $\mu$, $\la^{(0)}$, $\mu^{(0)}$, $\la^{(1)}$, and $\mu^{(1)}$ are as in either case above, then $d_{\la\mu}^{e,p}(1) = d_{\la^{(0)}\mu^{(0)}}^{e,p}(1) d_{\la^{(1)}\mu^{(1)}}^{e,p}(1)$.
\end{thmc}

\begin{rem}
In the special case of $r=1$, this was proved earlier, in \cite[Theorems~1 and 2]{j81}.
\end{rem}

\subsection{Scopes equivalences}\label{subsec:scopes}

In several places, we will make use of certain Morita equivalences between blocks of Hecke algebras, known as Scopes equivalences.
Scopes introduced these for the symmetric groups in~\cite{scopes}, and this theory was easily generalised to the Hecke algebras by Jost~\cite{jost}.

First, note that the definition of abacus that we are using is convenient for our purpose, but usually a truncated version of this abacus is used, where we do not allow the runners to extend infinitely upwards.
In truncating the abacus, we are essentially forcing an abacus display to have finitely many beads.
In this setting, it is common to then use beta-numbers $\beta_i = \la_i - i + r$, for $r\geq \la'_1$, to yield an abacus display with $r$ beads.
We may adjust $r$ in order to give a clean description of the Scopes equivalence.
So a given partition can have many different abacus displays.
When applying the Scopes equivalence, we number the runners of a chosen abacus display $0, 1, \dots, e-1$, so that position $i$ is on runner $i$ for each $i=0,1,\dots,e-1$.

Suppose that $B=B(\rho,w)$ is a block of $\hhh$
, and that for some $i$ an abacus display for $\rho$ (or equivalently, for any partition in $B$) has $k$ more beads on runner $i$ than on runner $i-1$, for some $k\geq w$.
Let $A$ be the block of $\hhh[n-k]$ of weight $w$ and core $\Phi(\rho)$, whose abacus display is obtained from that of $\rho$ by swapping runners $i$ and $i-1$.
We may define this map $\Phi$ in the same way for any partition $\la \in B$.
That is, we swap runners $i$ and $i-1$ of the corresponding abacus display for $\la$, yielding a partition $\Phi(\la) \in A$.

If $i=0$, we actually want to swap runners $0$ and $e-1$, and in doing so we need $k+1$ more beads on runner $0$ than runner $e-1$.
We will favour changing $r$ to avoid the need for this exceptional treatment.

For example, taking (core) partitions $(2^2,1^2)$ and $(2,1^2)$, with $e=3$, we have the following two abacus displays and their truncations.
Choosing $r=9$, we have
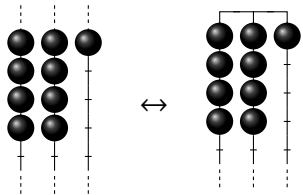
\[
\abacus(vvv,bbb,bbn,bbn,bbn,nnn,vvv)
\quad \leftrightarrow \quad
\abacus(lmr,bbb,bbn,bbn,bbn,nnn,vvv)
\]
while $r=10$ gives us the below abacus display, for which we may apply the map $\Phi$ as depicted.
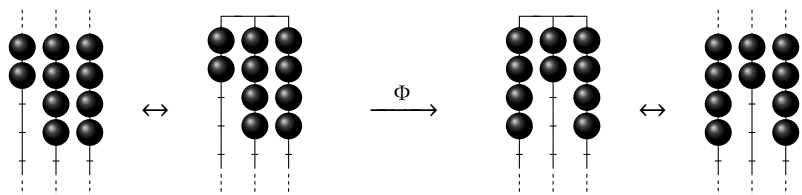
\[
\abacus(vvv,bbb,bbb,nbb,nbb,nnn,vvv)
\quad \leftrightarrow \quad
\abacus(lmr,bbb,bbb,nbb,nbb,nnn,vvv)
\qquad \xrightarrow[]{\phantom{aa} \Phi \phantom{aa}} \qquad
\abacus(lmr,bbb,bbb,bnb,bnb,nnn,vvv)
\quad \leftrightarrow \quad
\abacus(vvv,bbb,bbb,bnb,bnb,nnn,vvv)
\]

Scopes showed that this $\Phi$ is a bijection from $B$ to $A$, that maps $e$-regular partitions to $e$-regular partitions.
We say that such a pair of blocks \emph{forms a $[w:k]$ pair}.
Note that many authors do not assume that $k\geq w$ when using this terminology, unlike in the present paper.

\begin{thmc}{jost}{Theorem 7.3}\label{thm:scopes}
If $B$ and $A$ are blocks of $\hhh$ and $\hhh[n-k]$ as above, forming a $[w:k]$ pair for $k\geq w$, then they are Morita equivalent via $\D\la \leftrightarrow \D{\Phi(\la)}$ and $\spe\la \leftrightarrow \spe{\Phi(\la)}$.
In particular, if $\la$ and $\mu$ are partitions of $n$, with $\mu$ $e$-regular, then for any $p\geq0$, $d_{\la\mu}^{e,p}(1) = d_{\Phi(\la)\Phi(\mu)}^{e,p}(1)$.
\end{thmc}

\subsection{Reduction theorem and Jantzen filtration}\label{subsec:reduction}

\begin{prop}\label{reduction}
Let $A$ be a Schurian-finite algebra.
\begin{enumerate}[label=(\roman*)]
\item
If $B$ is a factor algebra of $A$, then $B$ is Schurian-finite.
\item
If $B=eAe$, for an idempotent $e\in A$,  then $B$ is Schurian-finite.
\end{enumerate}
\end{prop}
\begin{proof}
(i) is obvious.
(ii) follows from \cite[Proposition~2.4]{Ad2}.
\end{proof}

\begin{cor}\label{cor:Schurianinfinitequiver}
If the Gabriel quiver of a finite-dimensional algebra $A$ over $\bbf$ contains the quiver of an affine Dynkin diagram with zigzag orientation (i.e.~such that every vertex is a sink or a source) as a subquiver, then $A$ is Schurian-infinite.
\end{cor}
\begin{proof}
We may assume that $A$ is a basic algebra without loss of generality.
Then, $A\cong \bbf Q/I$, for a finite quiver $Q$ and an admissible ideal $I\subseteq \bbf Q$.
Let $e\in A$ be the sum of the idempotents associated with the vertices of the subquiver.
Then, since the path algebra of the subquiver we consider here is radical square zero, we have a surjective algebra homomorphism from $eAe$ to the path algebra of the subquiver.
The latter has preprojective and preinjective components with infinitely many vertices, by \cite[Chap.~VIII, Corollary~2.3]{ASS1}, each of which are Schurian, by \cite[Chap.~VIII, Lemma~2.7]{ASS1}.
Finally, \cref{reduction} implies that $A$ is Schurian-infinite.
\end{proof}

The following is Shan's theorem \cite[Theorem~0.1]{ShanJantzenfilt}.

\begin{thmc}{ShanJantzenfilt}{Theorem~0.1}\label{thm:Shan}
Suppose that $q=\exp(-2\pi i/e)\in \bbc$ with $e\geq 3$.
Let $\la, \mu$ be partitions of $n$ and the modules 
$W(\la')$ and $L(\mu')$ over the $q$-Schur algebra $S_q(n,n)$ are the Weyl module with highest weight $\la'$ and the irreducible module with highest weight $\mu'$, respectively.
Then, the graded decomposition numbers are given by the Jantzen filtration of $W(\la')$ as follows.
\[
d^{e,0}_{\la\mu}(v) = \sum_{i\geq 0} [J^iW(\la')/J^{i+1}W(\la'):L(\mu')] v^i
\]
\end{thmc}

Furthermore, as is pointed out in \cite[Remark~6.6]{ShanJantzenfilt}, the radical filtration of $W(\la')$ coincides with the Jantzen filtration 
by the dual statement of \cite[Lemma~5.2.2]{BB93} and \cite[Proposition~5.5]{ShanJantzenfilt} which proves that the geometric Jantzen filtration gives rise to the Jantzen filtration of $W(\la')$.
The assumption of \cite[Lemma~5.2.2]{BB93} for finite-dimensional Schubert varieties is known to hold.
In particular, if $[J^1W(\la')/J^{2}W(\la'):L(\mu')] \neq 0$, then $W(\la')$ has a quotient that is a uniserial module of length two whose head is $L(\la')$ and whose socle is $L(\mu')$.

\begin{lem}\label{lem:uniserial in char=0}
Suppose that $q=\exp(-2\pi i/e)\in \bbc$ with $e\geq 3$ as above.
If $\la, \mu$ are $e$-regular partitions of $n$ and the coefficient of $v$ in $d^{e,0}_{\la\mu}(v)$ is nonzero,
then $\Ext^1(\D\la, \D\mu) = \Ext^1(\D\mu,\D\la)\neq 0$.
\end{lem}

\begin{proof}
By the remark after \cref{thm:Shan}, $L(\mu')$ appears in $\Rad(W(\la'))/\Rad^2(W(\la'))$.
Since $\mu'$ is $e$-restricted, the Schur functor sends $W(\lambda')$ to the dual Specht module $\rrspe{\la'} \cong \spe\la\otimes\sgn$ and $L(\mu')$ to $\D{\mu}\otimes\sgn$.
That $\D\la$ is the unique head of $\spe\la$ implies that $\D\mu$ appears in the second layer of the radical series of $\spe\la$.
\end{proof}

\begin{prop}\label{prop:matrixtrick}
Suppose that $e\geq3$ and $\bbf$ has characteristic $p\geq 0$.
If a submatrix of the graded decomposition matrix in characteristic $0$ is one of the following matrices, and $d_{\la\mu}^{e,p}(1) = d_{\la\mu}^{e,0}(1)$ holds, for all $e$-regular partitions $\la, \mu$ that label rows of the submatrix, then the block in which those partitions belong is Schurian-infinite.

\begin{multicols}{2}
\begin{equation*}\label{targetmatrix}\tag{\(\dag\)}
\begin{pmatrix}
1\\
v & 1\\
0 & v & 1\\
v & v^2 & v & 1
\end{pmatrix}
\end{equation*}
\begin{equation*}\label{targetmatrix2}\tag{\(\dag'\)}
\begin{pmatrix}
1\\
v & 1\\
v^2 & v & 1\\
v & 0 & v & 1
\end{pmatrix}
\end{equation*}
\begin{equation*}\label{targetmatrix3}\tag{\(\dag''\)}
\begin{pmatrix}
1\\
v & 1\\
v^2 & v & 1\\
v & v^2 & v & 1
\end{pmatrix}
\end{equation*}
\begin{equation*}\label{targetmatrixalt}\tag{\(\ddag\)}
\begin{pmatrix}
1\\
v & 1\\
v & 0 & 1\\
v^2 & v & v & 1
\end{pmatrix}
\end{equation*}
\begin{equation*}\label{targetmatrixaltsquare}\tag{\(\clubsuit\)}
\begin{pmatrix}
1\\
0 & 1\\
v & v & 1\\
v & v & 0 & 1
\end{pmatrix}
\end{equation*}
\begin{equation*}\label{targetmatrixstar}\tag{\(\spadesuit\)}
\begin{pmatrix}
1\\
0 & 1\\
v & v & 1\\
0 & v^2 & v & 1\\
v^2 & 0 & v & 0 &1
\end{pmatrix}
\end{equation*}
\end{multicols}
\end{prop}
\begin{proof}
Let $B$ denote the block in characteristic $0$ in which the four or the five partitions that label the rows and columns of the submatrix belong, and let $B_\bbf$ denote the corresponding block in characteristic $p$.
Denote the $e$-regular partitions by $\la^{(i)}$, for $1\leq i\leq 4$ or $1\leq i\leq 5$.
Then, \cref{lem:uniserial in char=0} implies that we have an $A^{(1)}_3$ quiver (square) or a $D^{(1)}_4$ quiver (4-pointed star) with zigzag orientation as a subquiver of the Gabriel quiver of $B$.
By \cref{lem:charfree}, $d_{\la\mu}^{e,p}(v) = d_{\la\mu}^{e,0}(v)$.
Let $\D\mu$ and $\D\mu_\bbf$ be the simple $B$-module and the simple $B_\bbf$-module labeled by $\mu$, respectively.
We denote by $\spe\la_\bbf$ the Specht $B_\bbf$-module labeled by $\la$.
Now we consider the modular reduction of $\D\la/\Rad^2(\D\la)$ as a factor module of $\spe\la_\bbf$.
If $d_{\la\mu}^{e,0}(v) = v^2$, then $[\spe\la/\Rad^2(\spe\la):\D\mu] = 0$.
Then, $d_{\la\mu}^{e,p}(1) = d_{\la\mu}^{e,0}(1)\in \{0,1\}$ implies that $\D\mu_\bbf$ does not appear in the modular reduction of $\D\la/\Rad^2(\D\la)$ as a composition factor.

The same argument shows that $\D\la_\bbf$ appears with multiplicity $1$ as the unique head of the modular reduction of $\D\la/\Rad^2(\D\la)$, and $\D\mu_\bbf$ appears 
with multiplicity $1$ as one of the composition factors of the modular reduction if $d^{e,0}_{\la\mu}(v) = v$, 
and all the other composition factors of the modular reduction are $\D\nu_\bbf$ where $\nu$ is not among the four or five partitions $\la^{(i)}$.
Therefore, we may obtain an indecomposable $B_\bbf$-module that has the unique head $\D\la_\bbf$ and submodules $\D\mu_\bbf$ for $\mu$ with $d^{e,0}_{\la\mu}(v)=v$, and all the other composition factors of the module are $\D\nu_\bbf$ for partitions $\nu$ that are not among those we have labelled by $\la^{(i)}$.

Let $P_\bbf^\mu$ be the projective cover of $\D\mu_\bbf$ and let $t$ be the sum of the idempotents in the basic algebras of $B_\bbf$ that are projectors to $P_\bbf^{\la^{(i)}}$, summing over all $i$.
Then $t$ kills every simple module $\D\nu_\bbf$ that is not labelled by some $\la^{(i)}$.
Since $\Rad(t{B_\bbf}t)=t\Rad(B_\bbf)t$, it follows that we have 
$\Ext^1(t\D\la_\bbf, t\D\mu_\bbf)\ne0$, for simple $t{B_\bbf}t$-modules $t\D\la_\bbf$ and $t\D\mu_\bbf$ with $d^{e,0}_{\la\mu}(v) = v$.
This implies that the Gabriel quiver of $t{B_\bbf}t$ contains either $A^{(1)}_3$ or $D^{(1)}_4$ quiver with zigzag orientation, so that the Gabriel quiver of $B_\bbf$ contains one of them as a subquiver.
We apply \cref{cor:Schurianinfinitequiver} to conclude that $B_\bbf$ is Schurian-infinite.
\end{proof}

\section{Graded Scopes Equivalence}\label{sec:gradedScopes}

We consider the Scopes equivalence (c.f.~\cref{subsec:scopes}) in the graded setting.
Suppose that $B$ and $A$ form a $[w : k]$ pair, and consider its graded version, which we denote by  the cyclotomic KLR algebras $R^{\La_0}(\beta)$ and $R^{\La_0}(\beta-k\alpha_i)$, respectively.
Throughout this section we assume that either $p \nmid e$ or $p = e$, so that there are blocks of some Hecke algebras that $R^{\La_0}(\beta)$ and $R^{\La_0}(\beta-k\alpha_i)$ are isomorphic to.

We may define the Scopes restriction and induction functors in the graded setting to be the (graded) cyclotomic divided powers $e_i^{(k)}$ and $f_i^{(k)}$ (c.f.~\cite[Section 4.6]{bk09}), but we take a different approach.

Recall that if a pair of blocks $B$ and $A$ forms a $[w:k]$ pair, then the ungraded Scopes restriction functor is the composition of four functors \cite{jost}, and that the ungraded Scopes induction functor is its adjoint functor.
We may consider the graded version of the Scopes restriction functor.

Define $\mathcal{N}_k$ to be the graded algebra with generators $u_j$ with $\deg u_j=-2$, $1\leq j\leq k-1$, obeying $u_j^2=0$ and the braid relations.
Let $\bbf[x_1,\dots,x_k]$ be the polynomial ring with degree given by $\deg x_a=2$.
Let $e_j$, for $1\leq j\leq k$, be the elementary symmetric polynomials in $x_1,\dots,x_k$, 
and $I = (e_1,\dots,e_k)$ the ideal generated by them.
Then, $\mathcal{N}_k\otimes \bbf[x_1,\dots,x_k]$, with the relations
\[
u_jx_a = x_au_j\;(a\neq j,j+1), \quad u_jx_{j+1}-x_ju_j=1=x_{j+1}u_j-u_jx_j
\]
is a graded algebra that is isomorphic to $R(k\alpha_i)$.
We denote by $Y$ the $R(k\alpha_i)$-module realised on $\bbf[x_1,\dots,x_k]/I$, where $x_a$ acts by multiplication and $u_j$ acts as the divided difference $\partial_i = (x_{i+1} - x_i)^{-1}(1-s_i)$.
The graded $R(k\alpha_i)$-module $Y$ is the unique irreducible graded $R(k\alpha_i)$-module up to shift.

Let $\hhh[k]$ be the subalgebra of $R(k\alpha_i)$ generated by $T_i =  u_i x_i - q x_i u_i$, for $1\leq i\leq k-1$.
We consider $\hhh[k]$ as a graded algebra concentrated in degree $0$, and may view any $R(k\alpha_i)$-module as a graded $\hhh[k]$-module by restriction.
Then, $Y$ is a graded $\hhh[k]$-module whose degree zero component is isomorphic to the sign module $\spe{(1^k)}$.
We denote by $P$ the projective cover of the $\hhh[k]$-module $\spe{(1^k)}$ concentrated in degree zero.

Let $e_{\beta-k\alpha_i, k\alpha_i}\in R^{\La_0}(\beta)$ be the sum of the idempotents $e(\res(T)*i^k)$, where $T$ runs through standard tableaux of partitions of $n-k$ which belong to the block $A$.
Then, for a graded $R^{\La_0}(\beta)$-module $M$, $e_{\beta-k\alpha_i,k\alpha_i} M$ is a graded $R^{\La_0}(\beta-k\alpha_i)\boxtimes \hhh[k]$-module.

\begin{defn}\label{def:grscopes}
The graded Scopes functor is the functor from the category of finite-dimensional graded $R^{\La_0}(\beta)$-modules to the category of  finite-dimensional graded $R^{\La_0}(\beta-k\alpha_i)$-modules given by
\[
M \longmapsto \Hom_{\hhh[k]}(P, e_{\beta-k\alpha_i,k\alpha_i} M),
\]
where homomorphism are degree preserving homomorphisms.
This is an exact functor. 
\end{defn}

Since $w\leq k$, every partition $\la$ that belongs to $R^{\La_0}(\beta)$ satisfies $\varepsilon_i(\la)=k$, as was proved in the proof of \cite[Lemma~2.1]{scopes}, so that the graded Scopes functor is a subfunctor of $e_i^{\rm max}=e_i^k$.
The correspondence $\la\mapsto \Phi(\la)$ is nothing but $\la\mapsto \tilde{e}_i^{\rm max}\la=\tilde{e}_i^{k}\la$.

\begin{prop}\label{prop:grScopes}
Suppose $B$ and $A$ form a $[w:k]$ pair.
\begin{enumerate}[label=(\roman*)]
\item
The graded Scopes functor sends graded Specht module $\spe\la$ to $\spe{\Phi(\la)}$, for partitions $\la$.

\item
The graded Scopes functor sends graded simple module $\D\la$ to $\D{\Phi(\la)}$, for $e$-regular partitions $\la$.

\item
If we forget the grading, the graded Scopes functor coincides with the ungraded Scopes functor defined in \cite{jost}.

\item
$d^{e,p}_{\la\mu}(v) = d^{e,p}_{\Phi(\la)\Phi(\mu)}(v)$, for any $p\geq0$.

\item
The category of finite-dimensional graded $B$-modules is category equivalent to the category of finite-dimensional graded $A$-modules.
\end{enumerate}
\end{prop}

\begin{proof}
\begin{enumerate}[label=(\roman*)]
\item
We see that $e_{\beta-k\alpha_i,k\alpha_i} \spe\la\cong \spe{\Phi(\la)}\otimes Y$ by~\cite[Corollary~5.8]{Mathas17}.
Then, $\Hom_{\hhh[k]}(P, Y) = \Hom_{\hhh[k]}(P, \spe{(1^k)}) = \bbf$ and 
\[
\Hom_{\hhh[k]}(P, \spe{\Phi(\la)}\otimes Y)\cong
\spe{\Phi(\la)}\otimes \Hom_{\hhh[k]}(P, Y)\cong \spe{\Phi(\la)}
\]
follows. 

\item
Note that the composition factors of $\spe\la$ are $\D{\mu}$ with $\mu \doms \la$, up to shift.
Since $\la\mapsto \Phi(\la)$ respects the lexicographic order~\cite[Lemma~2.2]{scopes}, and $d^{e,p}_{\la\mu}(1) = d^{e,p}_{\Phi(\la)\Phi(\mu)}(1)$, for any $p\geq0$ \cite[Lemma~5.3]{jost}, 
we have the result for the graded Scopes functor by using (i) and induction on the reverse lexicographic order.

\item
Recall the definition of Jost's Scopes functor.
After restricting an $\hhh$-module to $\hhh[n-k]\boxtimes \hhh[k]$, we tensor it with the ungraded right $\hhh[k]$-module ${\spe{(k)}}^*$.
He showed that the restriction of $\D\la$ is isomorphic to $\D{\Phi(\la)}\boxtimes \hhh[k]$ 
\cite[Corollary~6.3]{jost}.
Hence, the restricted module viewed as an $\hhh[k]$-module is free of finite rank.

\begin{itemize}
\item[(a)]
If we restrict an ungraded $R^{\La_0}(\beta)$-module to $\hhh[n-k]\boxtimes \hhh[k]$, and view it as an $R^{\La_0}(\beta-k\alpha_i)\boxtimes \hhh[k]$-module through the algebra homomorphism
\[
R^{\La_0}(\beta-k\alpha_i)\boxtimes \hhh[k] \hookrightarrow \hhh[n-k]\boxtimes \hhh[k] \hookrightarrow \hhh,
\]
each composition factor $\D\la$ changes to $\D{\Phi(\la)}\otimes \hhh[k]$, and tensoring it with ${\spe{(k)}}^*$ over $\hhh[k]$ has the effect that 
we replace $\hhh[k]$ with $\bbf$.

\item[(b)]
If we restrict a graded $R^{\La_0}(\beta)$-module to $R^{\La_0}(\beta-k\alpha)\boxtimes \hhh[k]$, each composition factor $\D{\la}\langle d\rangle$ changes to $\D{\Phi(\la)}\langle d\rangle\otimes Y$, and taking the space of degree preserving $\hhh[k]$-module homomorphisms from $P$ has the effect that we replace $Y$ with $\bbf$.
\end{itemize}
Comparing (a) and (b), we know that our graded Scopes functor coincides with Jost's, if we forget the grading.
More precisely, we prove the next result.

\begin{quote}
Let $\Phi_1$ be our functor defined in \cref{def:grscopes}, and let $\Phi_2$ be Jost's functor.
Then, ${\rm For}\circ\Phi_1\cong \Phi_2\circ {\rm For}$, where ${\rm For}$ is the forgetful functor.
\end{quote}

\medskip
For any $R^{\La_0}(\beta)$-module $M$, we have the inclusion of graded $R^{\La_0}(\beta-k\alpha_i)$-modules
\[
\Phi_1(M) = (e_{\beta-k\alpha_i,k\alpha_i}M)_0 \longrightarrow M,
\]
where $(e_{\beta-k\alpha_i,k\alpha_i}M)_0$ is the degree $0$ part of $e_{\beta-k\alpha_i,k\alpha_i}M$ with respect to the graded $R(k\alpha_i)$-module structure.
We choose a composition series
\[
0=M_0\subset M_1\subset \cdots \subset M_l=M.
\]
Then, $0 \to M_{l-1} \to M \to \D{\la}\langle d\rangle \to 0$, for some $\la$ and $d$.
We show by induction on the length of modules that the inclusion above induces ${\rm For}\circ\Phi_1(M)\cong\Phi_2\circ {\rm For}(M)$.
Assume that ${\rm For}\circ\Phi_1(M_k)\cong \Phi_2\circ {\rm For}(M_k)$ holds for $1\leq k\leq l-1$, and we restrict the module $M$ to 
$R^{\La_0}(n-k)\boxtimes R(k)$.
Then, we have the commutative diagram of $\hhh[n-k]$-modules
\[
\xymatrix@C=1em{
0 \ar[rr] &&  {\rm For}(e_{\beta-k\alpha_i,k\alpha_i}M_{l-1})_0 \ar[rr] \ar[d] &&  {\rm For}(e_{\beta-k\alpha_i,k\alpha_i}M)_0 \ar[rr] \ar[d] && {\rm For}(e_{\beta-k\alpha_i,k\alpha_i}\D\la\langle d\rangle)_0 \ar[rr] \ar[d] && 0 \\
0 \ar[rr] &&  {\rm For}(M_{l-1}) \ar[rr] \ar[d] &&  {\rm For}(M) \ar[rr] \ar[d] &&  {\rm For}(\D{\la}) \ar[rr] \ar[d] && 0 \\
&& {\spe{(k)}}^*\otimes_{\hhh[k]}  {\rm For}(M_{l-1}) \ar[rr]  \ar[d] && {\spe{(k)}}^*\otimes_{\hhh[k]} {\rm For}(M) \ar[rr] \ar[d] && {S^{(k)}}^*\otimes_{\hhh[k]} ({\rm For}(\D{\Phi(\la)})\boxtimes \hhh[k]) \ar[rr]  \ar[d] && 0 \\
0 \ar[rr] && \Phi_2( {\rm For}(M_{l-1})) \ar[rr] && \Phi_2({\rm For}(M)) \ar[rr] && \Phi_2( {\rm For}(\D\la\langle d\rangle)) \ar[rr] && 0
}
\]
where the first vertical arrow is the inclusion, the third vertical arrow is the block truncation.
Note that ${\rm For}(\D\la)$ restricts to 
${\rm For}(\D{\Phi(\la)})\boxtimes \hhh[k]$.
Hence, we have the commutative diagram
\[
\xymatrix@C=1em{
0 \ar[rr] &&  {\rm For}(\Phi_1(M_{l-1})) \ar[rr] \ar[d] &&  {\rm For}(\Phi_1(M)) \ar[rr] \ar[d] && {\rm For}(\Phi_1(\D\la\langle d\rangle)) \ar[rr] \ar[d] && 0 \\
0 \ar[rr] && \Phi_2( {\rm For}(M_{l-1})) \ar[rr] && \Phi_2({\rm For}(M)) \ar[rr] && \Phi_2( {\rm For}(\D\la\langle d\rangle)) \ar[rr] && 0
}
\]
such that the left vertical arrow and the right vertical arrow are isomorphisms by the induction hypothesis.
By the five lemma, we deduce that ${\rm For}\circ\Phi_1(M)\cong \Phi_2\circ{\rm For}(M)$.

\item
Since $d^{e,p}_{\la\mu}(1) = d^{e,p}_{\Phi(\la)\Phi(\mu)}(1)$, for any $p\geq0$, this follows from (i) and (ii).

\item
We have to show that the graded Scopes functor is fully faithful and dense.
Since the ungraded Scopes functor induces category equivalence, (iii) implies the fully faithfulness.
To see that it is dense, we argue by induction on the length of graded modules.
Suppose that
\[
0 \rightarrow \D{\Phi(\la)} \rightarrow M \rightarrow N \rightarrow 0
\]
and that $L$ maps to $N$ under the graded Scopes functor.
Choose the element in \linebreak
$\EXT^1(N,\D{\Phi(\la)})$ which represents this extension.

Since $\Ext^1(N,D^{\Phi(\la)})\cong \Ext^1(L,D^\la)$ by the ungraded Scopes equivalence~\cite[Theorem~7.3]{jost}, 
we may choose the corresponding element in $\EXT^1(L,D^\la)$ which gives a short exact sequence of graded modules
\[
0 \rightarrow D^\la \rightarrow X \rightarrow L \rightarrow 0.
\]
Then, it induces $0 \rightarrow D^{\Phi(\la)} \rightarrow X' \rightarrow N \rightarrow 0$, 
where $X'$ is the image of $X$ under the graded Scopes functor, and we have an isomorphism $X'\cong M$ if we forget the grading.
This isomorphism is a sum of homomorphisms of various degrees.
However, since the homomorphisms in the short exact sequences are all degree $0$ homomorphisms, this isomorphism is pure of degree $0$.
Hence the graded Scopes functor is a dense functor. \qedhere
\end{enumerate}
\end{proof}

\begin{rem}
Since the indecomposable direct summands of the regular module of $R(k\alpha_i)$ are $\bbf[x_1,\dots,x_k]\langle -2j \rangle$, where $0\leq j\leq k(k-1)/2$, 
and 
\[
f_i^kN = R^{\La_0}(\beta)e_{\beta-k\alpha_i,k\alpha_i}\otimes_{R^{\La_0}(\beta-k\alpha_i)} N
\cong R^{\La_0}(\beta)e_{\beta-k\alpha_i,k\alpha_i}\otimes_{R^{\La_0}(\beta-k\alpha_i)\boxtimes R(k\alpha_i)} N\otimes R(k\alpha_i),
\]
we may define the graded Scopes induction functor by
\[
N \longmapsto R^{\La_0}(\beta)e_{\beta-k\alpha_i,k\alpha_i}\otimes_{R^{\La_0}(\beta-k\alpha_i)\boxtimes R(k\alpha_i)} N\otimes \bbf[x_1,\dots,x_k]\langle -k(k-1)\rangle.
\]
This is an exact functor.
Moreover, \cite[Corollary~4.7]{hm12} implies that the graded Scopes functor sends $\spe{\Phi(\la)}$ to $\spe\la$.

Unlike the Scopes restriction functor, we may generalise Scopes and Jost's induction functor in a straightforward manner.
Namely, we may define
\[
N \longmapsto R^{\La_0}(\beta)e_{\beta-k\alpha_i,k\alpha_i}\otimes_{R^{\La_0}(\beta-k\alpha_i)\boxtimes R(k)} N\otimes \spe{(k)}.
\]
We have not checked whether it coincides with our definition (up to shift) or not.
\end{rem}

\section{Weight 2 blocks}\label{sec:wt2}

In this section, we will prove \cref{thm:main} for weight $2$ blocks.
We will follow the notations and conventions of the note of Fayers~\cite{fayerswt2data}.
This will allow us to work in a great deal of generality.
As we noted in the remarks before the statement of \cref{thm:main}, if $e=2$ and $p\neq 2$, weight $2$ blocks of the Hecke algebra are known to be Schurian-finite, and we are unable to handle the case $e=p=2$.
So we will assume throughout this paper that $e\geq 3$.
Originally, we were able to prove \cref{thm:main} for weight 2 blocks when $p\neq 2$ by applying the results of~\cite{fayerswt2data}, but the proof we present here extends more readily into the weight $3$ case.

By~\cref{thm:jamesconj}, when $w=2$, $d_{\la\mu}^{e,0}(v) = d_{\la\mu}^{e,p}(v)$ for all $p\neq 2$, so we may rely on characteristic 0 results throughout this section, unless $p=2$.
Thus, we also assume that $p\neq 2$ in this section, except where explicitly stated, but otherwise we allow $e$ and $p$ to be arbitrary.

We work across five separate cases, depending on the abacus for the core $\rho$ indexing a given block $B(\rho,2)$, finding the submatrices, (\ref{targetmatrix}), (\ref{targetmatrixalt}), or (\ref{targetmatrixstar}) from \cref{prop:matrixtrick} in each case.
The runner-removal result in \cref{thm:runnerrem} will be key to reducing our necessary computations to known small rank cases.

\subsection{$p_{e-1} - p_{e-3} < e$}\label{subsec:wt2firstcase}

First, suppose that $\rho$ satisfies $p_{e-1} - p_{e-3} < e$, and $p\neq2$.

We will obtain the $4\times4$ submatrix of the graded decomposition matrix for the block, with rows and columns indexed by the partitions $\langle e-1\rangle$, $\langle e-2\rangle$, $\langle e-3\rangle$, and $\langle e-3, e-1\rangle$, and see that it is (\ref{targetmatrix}).

It is clear that for these four partitions, we may apply the runner removal result \cref{thm:runnerrem} to remove all but runners $e-3$, $e-2$, and $e-1$, leaving the abacus configuration of the core as one of those depicted below.
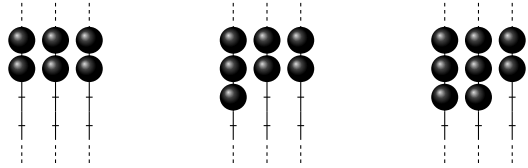
\[
\abacus(vvv,bbb,bbb,nnn,nnn,vvv) \qquad\qquad \abacus(vvv,bbb,bbb,bnn,nnn,vvv) \qquad\qquad \abacus(vvv,bbb,bbb,bbn,nnn,vvv)
\]
We place $p_{e-3}$ on the leftmost runner.
Then,
we may, without loss of generality, work with the first display above.
The corresponding four partitions then become $(6)$, $(5,1)$, $(4,1^2)$, and $(3,2,1)$, from which we may deduce that the corresponding submatrix (for $e=3$) is known to be (\ref{targetmatrix}) -- for instance, see \cite[Appendix~B]{mathas}.
The result now follows.\\

%

Now, if $p=2$, we may directly apply \cref{thm:wt2ext}; note that we do so for our four partitions on the $e$-runner abacus, as we cannot appeal to a runner removal result here.
None of our partitions are of the form $\langle i^2 \rangle$ for any $i$, and $\la$ and $\mu$ are adjacent whenever $d_{\la\mu}^{e,0}(v) = v$, so that we get the same subquiver of the Gabriel quiver of $B$ in characteristic $2$ as characteristic $0$ -- i.e.~we have an $A^{(1)}_3$ quiver (square) with zigzag orientation as a subquiver.

Alternatively, we may apply \cref{thm:wt2adjust}.
Since
\[
d_{\la\mu}^{e,2}(v) = d_{\la\mu}^{e,0}(v) + \sum_{\nu \domsby \mu} d_{\la\nu}^{e,0}(v) a_{\nu\mu}(v),
\]
and $d_{\la\nu}^{e,0}(v) = 0$ unless $\la \domby \nu$, it suffices to note that no $\langle e-3, e-1\rangle \domby \nu \domsby \langle e-1 \rangle$ can be of the form $\langle i^2 \rangle$, and thus that every summand on the right has either $d_{\la\nu}^{e,0}(v) = 0$ or $a_{\nu\mu}(v) = 0$.

In other words, the submatrix we found in other characteristics is identical in characteristic $2$, and we may apply \cref{prop:matrixtrick}.

\subsection{$p_{e-1} - p_{e-2} < e$ and $p_{e-2} - p_{e-3} < e$, but $p_{e-1} - p_{e-3} > e$}\label{subsec:wt2secondcase}

Next, suppose that $\rho$ satisfies $p_{e-1} - p_{e-2} < e$ and $p_{e-2} - p_{e-3} < e$, but $p_{e-1} - p_{e-3} > e$.

We will obtain the $4\times4$ submatrix of the graded decomposition matrix for the block, with rows and columns indexed by the partitions $\langle e-1\rangle$, $\langle e-2\rangle$, $\langle e-2, e-1\rangle$, and $\langle (e-1)^2\rangle$, and see that it is (\ref{targetmatrix}).

As in \cref{subsec:wt2firstcase}, it is clear that for these four partitions, we may apply the runner removal result \cref{thm:runnerrem} to remove all but runners $e-3$, $e-2$, and $e-1$, leaving the abacus configuration of the core as depicted below, where we place $p_{e-3}$ on the leftmost runner.
Then it is easy to see that $p_{e-2}$ cannot be on the middle runner.
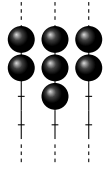
\[
\abacus(vvv,bbb,bbb,nbn,nnn,vvv)
\]
The corresponding four partitions then become $(7)$, $(5,2)$, $(4,3)$, and $(4,2,1)$, from which we may deduce that the corresponding submatrix (for $e=3$) is known to be (\ref{targetmatrix}) -- for instance, see \cite[Appendix~B]{mathas}.
The result now follows.\\

%

Now, if $p=2$, we may directly apply \cref{thm:wt2ext} as in \cref{subsec:wt2firstcase} -- only one of our partitions is of the form $\langle i^2 \rangle$ for any $i$ (that is, when $i=e-1$), but we do not have $_{e-2}0_{e-1}$.
Again, $\la$ and $\mu$ are adjacent whenever $d_{\la\mu}^{e,0}(v) = v$, so that we get the same subquiver of the Gabriel quiver in characteristic $2$ as characteristic $0$ -- i.e.~an $A^{(1)}_3$ quiver (square) with zigzag orientation.

Alternatively, we may apply \cref{thm:wt2adjust} as in \cref{subsec:wt2firstcase}.
Now, note that the only $\langle (e-1)^2 \rangle \domby \nu \domsby \langle e-1 \rangle$ of the form $\langle i^2 \rangle$ is $\langle (e-1)^2 \rangle$, and thus that every summand on the right has either $d_{\la\nu}^{e,0}(v) = 0$ or $a_{\nu\mu}(v) = 0$, except the summand with $\nu = \langle (e-1)^2 \rangle$, which only occurs if $\la = \langle (e-1)^2 \rangle$ to.
In this case, if $\mu$ is any one of our four possible partitions, the above becomes
\begin{align*}
d_{\langle (e-1)^2 \rangle\mu}^{e,2}(v) &= d_{\langle (e-1)^2 \rangle\mu}^{e,0}(v) + d_{\langle (e-1)^2 \rangle \langle (e-1)^2 \rangle}^{e,0}(v) a_{\langle (e-1)^2 \rangle \mu}(v)\\
&= d_{\langle (e-1)^2 \rangle\mu}^{e,0}(v) + a_{\langle (e-1)^2 \rangle \mu}(v)\\
&= d_{\langle (e-1)^2 \rangle\mu}^{e,0}(v) \text{ by \cref{thm:wt2adjust}, as $p_{e-1}-p_{e-2}<e$, or $_{e-2}1_{e-1}$}.
\end{align*}
In other words, the submatrix we found in other characteristics is again identical in characteristic $2$.

\subsection{$p_{e-1} - p_{e-2} > e$ and $p_{e-2} - p_{e-3} < e$}\label{subsec:wt2thirdcase}

Next, suppose that $\rho$ satisfies $p_{e-1} - p_{e-2} > e$ and $p_{e-2} - p_{e-3} < e$.

We will obtain the $4\times4$ submatrix of the graded decomposition matrix for the block, with rows and columns indexed by the partitions $\langle e-1\rangle$, $\langle e-2, e-1\rangle$, $\langle e-2\rangle$, and $\langle e-3\rangle$, and see that it is (\ref{targetmatrix}).

As in \cref{subsec:wt2firstcase,subsec:wt2secondcase}, it is clear that for these four partitions, we may apply the runner removal result \cref{thm:runnerrem} to remove all but runners $e-3$, $e-2$, and $e-1$, leaving abacus displays with $e=3$.
We must have either $p_{e-2} = p_{e-3} + 1$ and $p_{e-1} = p_{e-2} + (k-1)e + 1$ or $p_{e-2} = p_{e-3} + 2$ and $p_{e-1} = p_{e-2} + (k-1)e + 2$, for some $k\geq 2$.
The corresponding $k+3$ bead abacus displays are depicted below.
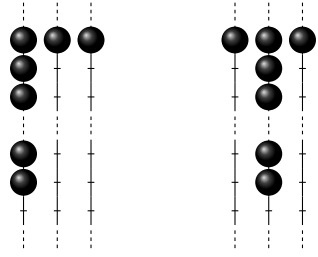
\[
\abacus(vvv,bbb,bnn,bnn,vvv,bnn,bnn,nnn,vvv)
\qquad\qquad
\abacus(vvv,bbb,nbn,nbn,vvv,nbn,nbn,nnn,vvv)
\]
Here, the diagrams have $k$ more beads on the first and second runners, respectively, than on the other two.
Because of this, swapping the $0$- and $1$-runners (i.e.~the leftmost and middle runners) or the $1$- and $2$-runners (i.e.~the middle and rightmost runners) always involves swapping runners on which the number of beads differ by at least $w=2$.

For $k\geq 2$, we let $\rho^{(2k-3)} := (2k-2, 2k-4,\dots, 2)$ denote the partition with abacus display on the left above (satisfying $p_{e-2} = p_{e-3} + 1$ and $p_{e-1} = p_{e-2} + (k-1)e + 1$) and $\rho^{(2k-2)} := (2k-1,2k-3,\dots,1)$ denote the partition with abacus display on the right above (satisfying $p_{e-2} = p_{e-3} + 2$ and $p_{e-1} = p_{e-2} + (k-1)e + 2$)
Then our first few partitions are $\rho^{(1)}=(2)$, $\rho^{(2)}=(3,1)$, $\rho^{(3)}=(4,2)$, $\rho^{(4)}=(5,3,1)$, and $\rho^{(5)}=(6,4,2)$.
It is clear that swapping the $0$- and $1$-runners swaps $\rho^{(2k-3)}$ and $\rho^{(2k-2)}$, while swapping the $1$- and $2$-runners of the above abacus display for $\rho^{(2k-2)}$ yields a $k+3$ bead abacus display for $\rho^{(2k-1)}$.
But $\rho^{(2k-1)}$ has a $(k+1)+3$ bead abacus display as above, so that we can again swap the $0$- and $1$-runners to obtain $\rho^{(2k)}$, and so on.
Then we see that all possible cores in this case are among our partitions $\rho^{(i)}$, and they may be recursively obtained from $\rho^{(1)} = (2)$ by swapping adjacent runners, where one has $k\geq 2 = w$ more beads than the others.
So we may apply \cref{prop:grScopes}, and reduce our computation to the case of the block with core $(2)$.

Notice that when we swap the $0$- and $1$-runners to interchange $\rho^{(2k-3)}$ and $\rho^{(2k-2)}$, this corresponds to performing $(1-k)$-induction or $(1-k)$-restriction on the core (where we take the residue of $1-k$ modulo $3$).
Similarly, when swapping the $1$- and $2$-runners to interchange $\rho^{(2k-2)}$ and $\rho^{(2k-1)}$, this corresponds to performing $(2-k)$-induction or $(2-k)$-restriction on the core.

Thus, we may alternatively characterise our runner swaps above as applying (maximal powers of) Kashiwara operators recursively to the cores, and the partitions in those blocks, as follows.
For $i \geq 0$,
\begin{align*}\label{eq:inducecase3cores}
\begin{split}
\rho^{(6i+2)} = \tilde{f_2}^{3i+2} \rho^{(6i+1)}, \quad \rho^{(6i+3)} &= \tilde{f_0}^{3i+2} \rho^{(6i+2)}, \quad \rho^{(6i+4)} = \tilde{f_1}^{3i+3} \rho^{(6i+3)},\\
\rho^{(6i+5)} = \tilde{f_2}^{3i+3} \rho^{(6i+4)}, \quad \rho^{(6i+6)} &= \tilde{f_0}^{3i+4} \rho^{(6i+5)}, \quad \rho^{(6i+7)} = \tilde{f_1}^{3i+4} \rho^{(6i+6)}.
\end{split}
\end{align*}
Since Scopes equivalences preserve our labelling of Specht modules and simple modules when we reduce to the block with core $(2)$, we have for this block that
$\langle e-1\rangle = (8)$, $\langle e-2, e-1\rangle = (5,2,1)$, $\langle e-2\rangle = (4,3,1)$, and $\langle e-3\rangle = (3^2,1^2)$, from which we may deduce that the corresponding submatrix is known to be (\ref{targetmatrix}) -- for instance, see \cite[Appendix~B]{mathas}.
The result now follows.\\

%

Now, if $p=2$, we may argue as in \cref{subsec:wt2firstcase}, by directly applying \cref{thm:wt2ext} -- none of our partitions are of the form $\langle i^2 \rangle$ for any $i$, and $\la$ and $\mu$ are adjacent whenever $d_{\la\mu}^{e,0}(v) = v$, so that we get the same subquiver of the Gabriel quiver in characteristic $2$ as characteristic $0$ -- i.e.~an $A^{(1)}_3$ quiver (square) with zigzag orientation.

\subsection{$p_{e-1} - p_{e-2} < e$ and $p_{e-2} - p_{e-3} > e$}\label{subsec:wt2fourthcase}

Next, suppose that $\rho$ satisfies $p_{e-1} - p_{e-2} < e$ and $p_{e-2} - p_{e-3} > e$.

We will obtain the $4\times4$ submatrix of the graded decomposition matrix for the block, with rows and columns indexed by the partitions $\langle e-1\rangle$, $\langle e-2\rangle$, $\langle e-2, e-1\rangle$, and $\langle (e-1)^2\rangle$, and see that it is (\ref{targetmatrix}).

This case proceeds analogously to the one handled in \cref{subsec:wt2thirdcase}.
As before, it is clear that for these four partitions, we may apply the runner removal result \cref{thm:runnerrem} to remove all but runners $e-3$, $e-2$, and $e-1$, leaving abacus displays with $e=3$.
We must have either $p_{e-2} = p_{e-3} + (k-1)e + 1$ and $p_{e-1} = p_{e-2} + 1$ or $p_{e-2} = p_{e-3} + (k-1)e + 2$ and $p_{e-1} = p_{e-2} + 2$, for some $k\geq 2$.
The corresponding $(2k+3)$ bead abacus displays are depicted below.
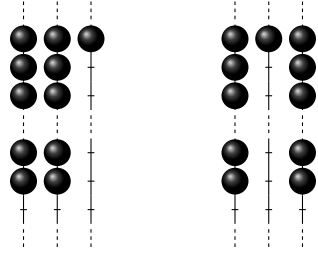
\[
\abacus(vvv,bbb,bbn,bbn,vvv,bbn,bbn,nnn,vvv)
\qquad\qquad
\abacus(vvv,bbb,bnb,bnb,vvv,bnb,bnb,nnn,vvv)
\]
Here, the diagrams have $k$ more beads on the leftmost and middle runners than the rightmost, or $k$ more beads on the leftmost and rightmost runners than the middle one, respectively.

Because of this, swapping the $1$- and $2$-runners (i.e.~the middle and rightmost runners) or the $0$- and $1$-runners (i.e.~the leftmost and middle runners) always involves swapping runners on which the number of beads differ by at least $w=2$.

For $k\geq 2$, we let $\sigma^{(2k-3)} = (\rho^{(2k-3)})' = ((k-1)^2, (k-2)^2,\dots, 1^2)$ denote the partition with abacus display on the left above (satisfying $p_{e-2} = p_{e-3} + (k-1)e + 1$ and $p_{e-1} = p_{e-2} + 1$) and $\sigma^{(2k-2)} = (\rho^{(2k-2)})' = (k,(k-1)^2, (k-2)^2,\dots, 1^2)$ denote the partition with abacus display on the right above (satisfying $p_{e-2} = p_{e-3} + (k-1)e + 2$ and $p_{e-1} = p_{e-2} + 2$).
Then our first few partitions are $\sigma^{(1)}=(1^2)$, $\sigma^{(2)}=(2,1^2)$, $\sigma^{(3)}=(2^1,1^2)$, $\sigma^{(4)}=(3,2^2,1^2)$, and $\sigma^{(5)}=(3^2,2^2,1^2)$.
It is clear that swapping the $1$- and $2$-runners swaps $\sigma^{(2k-3)}$ and $\sigma^{(2k-2)}$, while swapping the $0$- and $1$-runners of the above abacus display for $\sigma^{(2k-2)}$ yields a $2k+3$ bead abacus display for $\sigma^{(2k-1)}$.
But $\sigma^{(2k-1)}$ has a $2(k+1)+3$ bead abacus display as above, so that we can again swap the $0$- and $1$-runners to obtain $\sigma^{(2k)}$, and so on.
Then we see that all possible cores in this case are among our partitions $\sigma^{(i)}$, and they may be recursively obtained from $\sigma^{(1)} = (1^2)$ by swapping adjacent runners, where one has $k\geq 2 = w$ more beads than the others.
So we may apply \cref{prop:grScopes}, and reduce our computation to the case of the block with core $(1^2)$.

Notice that when we swap the $0$- and $1$-runners to interchange $\sigma^{(2k-3)}$ and $\sigma^{(2k-2)}$, this corresponds to performing $k$-induction on the core (where we take the residue of $k$ modulo $3$).
Similarly, when swapping the $1$- and $2$-runners to interchange $\sigma^{(2k-2)}$ and $\sigma^{(2k-1)}$, this corresponds to performing $(k+1)$-induction on the core.

Thus, we may alternatively characterise our runner swaps above as applying (maximal powers of) Kashiwara operators recursively to the cores, and the partitions in those blocks, as follows.
\begin{align*}\label{eq:inducecase4cores}
\begin{split}
\sigma^{(6i+2)} = \tilde{f_1}^{3i+2} \sigma^{(6i+1)}, \quad \sigma^{(6i+3)} &= \tilde{f_0}^{3i+2} \sigma^{(6i+2)}, \quad \sigma^{(6i+4)} = \tilde{f_2}^{3i+3} \sigma^{(6i+3)},\\
\sigma^{(6i+5)} = \tilde{f_1}^{3i+3} \sigma^{(6i+4)}, \quad \sigma^{(6i+6)} &= \tilde{f_0}^{3i+4} \sigma^{(6i+5)}, \quad \sigma^{(6i+7)} = \tilde{f_2}^{3i+4} \sigma^{(6i+6)}.
\end{split}
\end{align*}
Noting, as before, that Scopes equivalences preserve our labelling of Specht modules and simple modules when we reduce to the block with core $(1^2)$, we have for this block that $\langle e-1\rangle = (7,1)$, $\langle e-2\rangle = (6,2)$, $\langle e-2, e-1\rangle = (4^2)$, and $\langle (e-1)^2\rangle = (4,2^2)$, from which we may deduce that the corresponding submatrix is known to be (\ref{targetmatrix}) -- for instance, see \cite[Appendix~B]{mathas}.
The result now follows.\\


Now, if $p=2$, the exact same argument we used in \cref{subsec:wt2secondcase} shows that we get the same subquiver of the Gabriel quiver in characteristic $2$ as characteristic $0$ -- i.e.~an $A^{(1)}_3$ quiver (square) with zigzag orientation.

\subsection{$p_{e-1} - p_{e-2} > e$ and $p_{e-2} - p_{e-3} > e$}\label{subsec:wt2fifthcase}

Finally, suppose that $\rho$ satisfies $p_{e-1} - p_{e-2} > e$ and $p_{e-2} - p_{e-3} > e$.

We will obtain the $5\times5$ submatrix of the graded decomposition matrix for the block, with rows and columns indexed by the partitions $\langle e-1\rangle$, $\langle (e-1)^2\rangle$, $\langle e-2, e-1\rangle$, $\langle e-2\rangle$, and $\langle (e-2)^2\rangle$, and see that it is (\ref{targetmatrixstar}).

This case proceeds analogously to the one handled in \cref{subsec:wt2thirdcase}.
As before, it is clear that for these five partitions, we may apply the runner removal result \cref{thm:runnerrem} to remove all but runners $e-3$, $e-2$, and $e-1$, leaving abacus displays with $e=3$.
Then we have either $p_{e-1} = p_{e-2} + ke + 1$ and $p_{e-2} = p_{e-3} + je + 1$ or $p_{e-1} = p_{e-2} + ke + 2$ and $p_{e-2} = p_{e-3} + je + 2$, for some $j, k \geq 1$.
In this way, we index the cores by the pair of integers $(ke+1,je+1)$ or $(ke+2, je+2)$, so we will denote such core by $\kappa_{(ke+1,je+1)}$ or $\kappa_{(ke+2,je+2)}$, as appropriate.
We will draw the $(k+2j+3)$ bead abacus display for $\kappa_{(ke+1,je+1)}$ as below.
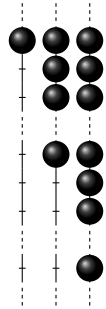
\[
\abacus(vvv,bbb,nbb,nbb,vvv,nbb,nnb,nnb,vvv,nnb,vvv)
\]
Note that the number of beads on the $0$- and $1$-runners differ by $j$, while those on the $1$- and $2$-runners differ by $k$.

Similarly, we draw the $(k+2j+5)$ bead abacus display for $\kappa_{(ke+2,je+2)}$ as below.
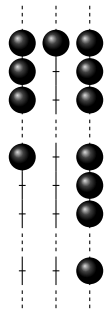
\[
\abacus(vvv,bbb,bnb,bnb,vvv,bnb,nnb,nnb,vvv,nnb,vvv)
\]
Note that the number of beads on the $0$- and $1$-runners differ by $j+1$, while those on the $1$- and $2$-runners differ by $k+j+1$.

It is easy to see that swapping the $1$- and $2$-runners on the above display
yields a $(k+2j+5)$ bead abacus display for $\kappa_{(ke+1,je+1)}$.
On the other hand, if we start with our $(k+2j+3)$ bead abacus display for $\kappa_{(ke+1,je+1)}$, swapping the $0$- and $1$-runners (which differ by $j$ beads) yields a $(k+2j+3)$ bead abacus display for $\kappa_{(ke+2,(j-1)e+2)}$, while 
swapping the $1$- and $2$-runners (which differ by $k$ beads) yields a $(k+2j+3)$ bead abacus display for $\kappa_{((k-1)e+2,je+2)}$.

Thus we may use the graded Scopes equivalences of \cref{prop:grScopes} to pass between the blocks with these cores, as in \cref{subsec:wt2thirdcase,subsec:wt2fourthcase}.
Noting, as before, that Scopes equivalences preserve our labelling of Specht modules and simple modules, it is clear that we can use Scopes equivalences to reduce all the way down to the block with core $\kappa_{(e+1,e+1)}$, for which the corresponding five partitions then become $(9,1^2)$, $(6,4,1)$, $(6,3,2)$, $(5,4,2)$, and $(3^2,2^2,1)$, from which we may deduce that the corresponding submatrix is known to be (\ref{targetmatrixstar}) -- for instance, see \cite[Appendix~B]{mathas}.
Our result now follows.\\

Now, if $p=2$, we must use different partitions.
If $e=3$, then the block is labelled by the $3$-core $(3,1^2)$.
By applying \cref{thm:wt2ext} our Gabriel quiver is
\[
\begin{tikzcd}
	{\langle (e-1)^2 \rangle} & {\langle e-1 \rangle} & {\langle e-2,e-1 \rangle} & {\langle e-2 \rangle} & {\langle (e-2)^2 \rangle}
	\arrow[from=1-1, to=1-2, shift left=.9ex, "\alpha_1" above]
	\arrow[from=1-2, to=1-1, shift left=.9ex, "\beta_1" below]
	\arrow[from=1-2, to=1-3, shift left=.9ex, "\alpha_2" above]
	\arrow[from=1-3, to=1-2, shift left=.9ex, "\beta_2" below]
	\arrow[from=1-3, to=1-4, shift left=.9ex, "\alpha_3" above]
	\arrow[from=1-4, to=1-3, shift left=.9ex, "\beta_3" below]
	\arrow[from=1-4, to=1-5, shift left=.9ex, "\alpha_4" above]
	\arrow[from=1-5, to=1-4, shift left=.9ex, "\beta_4" below]
\end{tikzcd}
\]
so that our methods are unable to determine the Schurian-infiniteness of the block.
We use a different method to prove that it is Schurian-infinite, so we will come back to this case at the end of the section.
Until then, we will assume that $e\geq4$.

If $_{e-4}0_{e-3}$, we take our partitions to be $\langle e-2, e-1 \rangle$, $\langle e-3, e-1 \rangle$, $\langle e-2 \rangle$, $\langle e-3, e-2 \rangle$.
Then \cref{thm:wt2ext} gives us that
\begin{multline*}
\Ext^1(\D{\langle e-2, e-1 \rangle}, \D{\langle e-3, e-1 \rangle}) \cong 
\Ext^1(\D{\langle e-2, e-1 \rangle}, \D{\langle e-2 \rangle})\\
\cong
\Ext^1(\D{\langle e-3, e-2 \rangle}, \D{\langle e-3, e-1 \rangle}) \cong 
\Ext^1(\D{\langle e-3, e-2 \rangle}, \D{\langle e-2 \rangle}) \cong \bbf,
\end{multline*}
provided we can show that the corresponding pairs of partitions are adjacent.
In order to do this, we should again show that these partitions give us (\ref{targetmatrixalt}) as the corresponding submatrix of the characteristic $0$ graded decomposition matrix -- in fact it suffices to pick out the four $v$ entries.
We may use \cref{thm:runnerrem} to reduce to the $e=4$ situation.
Then we may argue as for $p\neq 2$, but this time we must keep track of $p_3-p_2$, $p_2 - p_1$, and $p_1 - p_0$ -- we will again denote the corresponding core by $\kappa_{p_3-p_2, p_2 - p_1, p_1 - p_0}$.

Then since we can't ever have $p_i - p_j \equiv 0 \pmod 4$, it is not so difficult to check that all cores must fit into one of the following six families:
\[
\kappa_{ie+3, je+3, ke+3}, \; \kappa_{ie+2, je+3, ke+2}, \; \kappa_{ie+1, je+2, ke+3}, \; \kappa_{ie+3, je+2, ke+1}, \; \kappa_{ie+2, je+1, ke+2}, \text{ or } \kappa_{ie+1, je+1, ke+1},
\]
for $i,j,k \in \bbz_{>0}$.
We may check that, as before, we may perform runner swaps to get between these blocks, and that these will always involve runners whose number of beads differ by at least $2$, so that we may again apply \cref{prop:grScopes} to reduce down to the block with core $\kappa_{e+1, e+1, e+1} = (6,3^2,1^3)$.
In this block, we use the partitions $\langle e-2, e-1 \rangle = (10,6,4,1^3)$, $\langle e-3, e-1 \rangle = (10,3^3,2^2)$, $\langle e-2 \rangle = (9,7,4,1^3)$, and $\langle e-3, e-2 \rangle = (6^2,4,3,2^2)$, and the result follows.

To see the runner swaps explicitly, we use the Kashiwara operators $\tilde{e}_i^{\rm max}$ on the cores, so that it is presented more cleanly.
Then, taking subscripts modulo $e$, we have
\begin{itemize}
\item $\tilde{e}_{3i+2j+k+2}^{i+j+k+2} \kappa_{ie+3,je+3,ke+3} = \kappa_{ie+2, je+3, ke+2}$;

\item $\tilde{e}_{3i+2j+k+1}^{i+j+1} \kappa_{ie+2,je+3,ke+2} = \kappa_{ie+1, je+2, ke+3}$ and $\tilde{e}_{3i+2j+k-1}^{j+k+1} \kappa_{ie+2,je+3,ke+2} = \kappa_{ie+3, je+2, ke+1}$;

\item $\tilde{e}_{3i+2j+k-1}^{j+k+1} \kappa_{ie+1, je+2, ke+3} = \kappa_{ie+2, je+1, ke+2} = \tilde{e}_{3i+2j+k+1}^{i+j+1} \kappa_{ie+3, je+2, ke+1}$;

\item $\tilde{e}_{3i+2j+k}^{i} \kappa_{ie+1, je+2, ke+3} = \kappa_{(i-1)e+3, je+3, ke+3}$,
$\tilde{e}_{3i+2j+k-2}^{j} \kappa_{ie+2, je+1, ke+2} = \kappa_{ie+3, (j-1)e+3, ke+3}$,
and\linebreak
$\tilde{e}_{3i+2j+k}^{k} \kappa_{ie+3, je+2, ke+1} = \kappa_{ie+3, je+3, (k-1)e+3}$;

\item $\tilde{e}_{3i+2j+k}^{i+j+k+1} \kappa_{ie+2, je+1, ke+2} = \kappa_{ie+1, je+1, ke+1}$;

\item $\tilde{e}_{3i+2j+k-1}^{i} \kappa_{ie+1, je+1, ke+1} = \kappa_{(i-1)e+3, je+2, ke+1}$,
$\tilde{e}_{3i+2j+k-2}^{j} \kappa_{ie+1, je+1, ke+1} = \kappa_{ie+2, (j-1)e+3, ke+2}$,
and\linebreak
$\tilde{e}_{3i+2j+k-3}^{k} \kappa_{ie+1, je+1, ke+1} = \kappa_{ie+1, je+2, (k-1)e+3}$.
\end{itemize}
Each of these may be easily deduced by analysing the beta-numbers for these cores.

On the other hand, if $_{e-4}1_{e-3}$, we take our partitions to be $\langle e-2 \rangle$, $\langle e-3, e-2 \rangle$, $\langle e-3 \rangle$, $\langle e-4 \rangle$.
Then \cref{thm:wt2ext} gives us that
\[
\Ext^1(\D{\langle e-2 \rangle}, \D{\langle e-3, e-2 \rangle}) \cong \Ext^1(\D{\langle e-2 \rangle}, \D{\langle e-4 \rangle}) \cong \Ext^1(\D{\langle e-3 \rangle}, \D{\langle e-3, e-2 \rangle}) \cong \Ext^1(\D{\langle e-3 \rangle}, \D{\langle e-4 \rangle}) \cong \bbf,
\]
provided we can show that the corresponding pairs of partitions are adjacent.
This latter point is proved in an analogous manner to the previous case, essentially setting $k=0$ in the previous case.
Then Scopes equivalences reduce this to the case of the block with core $\kappa_{e+1,e+1,1} = (5,2^2)$, so that our four partitions are $\langle e-2 \rangle = (8,6,3)$, $\langle e-3, e-2 \rangle = (5^2,3,2,1^2)$, $\langle e-3 \rangle = (5,4,3^2,1^2)$, and $\langle e-4 \rangle = (5,3^3,1^3)$, giving the matrix (\ref{targetmatrix}).

In both cases, we obtain an $A^{(1)}_3$ quiver (square) with zigzag orientation as a subquiver of the Gabriel quiver.\\

Finally, we return to the case $e=3$, recalling that up to Scopes equivalence, we are looking at the block with core $(3,1^2)$.
As preparation, we first show that Specht modules in this block are all uniserial.

The simples in this block are labelled by the five partitions $\la_1 := \langle 2 \rangle = (9,1^2)$, $\la_2 := \langle 2^2 \rangle = (6,4,1)$, $\la_3 := \langle 1,2 \rangle = (6,3,2)$, $\la_4 := \langle 1 \rangle = (5,4,2)$ and $\la_5 := \langle 1^2 \rangle = (3^2,2^2,1)$.
The block also contains the $3$-singular partitions $\la_6 := (6,1^5)$, $\la_7 := (3^2,2,1^3)$, $\la_8 := (3,2^3,1^2)$, and $\la_9 := (3,1^8)$.
Using the LLT algorithm~\cite{LLT}, which computes decomposition numbers, by~\cite{ariki96} and \cref{thm:wt2adjust}, we may deduce that the (ungraded) decomposition matrix is as below.
\[
\begin{array}{c|ccccc}
 & \la_1 & \la_2 & \la_3 & \la_4 & \la_5 \\\hline
\la_1 & 1 & \cdot & \cdot & \cdot & \cdot \\
\la_2 & 1 & 1 & \cdot & \cdot & \cdot \\
\la_3 & 2 & 1 & 1 & \cdot & \cdot \\
\la_4 & 1 & 1 & 1 & 1 & \cdot \\
\la_5 & 1 & 0 & 1 & 1 & 1 \\
\la_6 & 0 & 0 & 1 & 0 & 0 \\
\la_7 & 0 & 0 & 1 & 2 & 1 \\
\la_8 & 0 & 0 & 0 & 1 & 1 \\
\la_9 & 0 & 0 & 0 & 1 & 0
\end{array}
\]
The Mullineux map swaps $\la_1$ and $\la_4$, $\la_2$ and $\la_5$, and fixes $\la_3$.

In the following, we use the terminology of string modules and band modules, for a general finite-dimensional algebra~\cite{stv21}.
However, as we use left modules, some care must be taken.
To define the string module $M(w)$, for a walk $w$ on the double quiver of the Gabriel quiver, we read the walk $w$ from right to left.
For example, $M(\beta_2\beta_1)$ is the uniserial module such that the head is $\D{\la_2}$, the heart is $\D{\la_1}$ and 
the socle is $\D{\la_3}$.

By the decomposition matrix, $\spe{\la_9}\cong \D{\la_4}$.
As $\spe{\la_8}$ is a submodule of $P^{\la_5}$ and $[\spe{\la_8}] = [\D{\la_4}] + [\D{\la_5}]$, $\spe{\la_8}$ is the uniserial module of length two whose head is $\D{\la_4}$ and whose socle is $\D{\la_5}$.
This is the string module $M(\beta_4)$.
Next we consider $\spe{\la_7}$.
As it is a submodule of $P^{\la_3}$, and $[S^{\la_7}/\Soc(\spe{\la_7})] = 2[\D{\la_4}] + [\D{\la_5}]$, $\Ext^1(\D{\la_4},\D{\la_4}) = 0$ and $\Ext^1(\D{\la_4},\D{\la_5})=\Ext^1(\D{\la_5},\D{\la_4}) = \bbf$, we must have that 
$\spe{\la_7}/\Soc(\spe{\la_7})\cong M(\alpha_4\beta_4)$, so that $\spe{\la_7}\cong M(\alpha_3\alpha_4\beta_4)$.
Hence
\[
\spe{\la_9}\cong \D{\la_4}, \quad \spe{\la_8}\cong M(\beta_4), \quad \spe{\la_7}\cong M(\alpha_3\alpha_4\beta_4).
\]
By similar arguments, we obtain
\begin{gather*}
\spe{\la_6}\cong \D{\la_3}, \quad \spe{\la_5}\cong M(\alpha_2\alpha_3\alpha_4), \quad \spe{\la_4}\cong M(\alpha_1\alpha_2\alpha_3), \\
\spe{\la_3}\cong M(\beta_1\alpha_1\alpha_2),\quad \spe{\la_2}\cong M(\beta_1), \quad \spe{\la_1}\cong \D{\la_1}.
\end{gather*}

We have Specht filtration $P^{\la_2} = V_0\supset V_1\supset V_2$ such that 
\[
V_0/V_1\cong \spe{\la_2}, \quad V_1/V_2\cong \spe{\la_3}, \quad V_2\cong \spe{\la_4}. 
\]
By the shape of the Gabriel quiver, the head and the socle of $\Rad(P^{\la_2})/\Soc(P^{\la_2})$ is $\D{\la_1}$ and its composition factors are $4[\D{\la_1}]+[\D{\la_2}]+2[\D{\la_3}]+[\D{\la_4}]$.
Since $\Rad(P^{\la_2})/\Soc(P^{\la_2})$ is self-dual, $\D{\la_i}$ appears in the head of $\Rad(P^{\la_2})/\Soc(P^{\la_2})$ only when 
$\D{\la_i}$ appears at least twice as the composition factor, so that $i=1$ or $i=3$, and we may rule out $i=3$ since $\Ext^1(\D{\la_2},\D{\la_3}) = 0$.
Hence the head and the socle of $\Rad(P^{\la_2})/\Soc(P^{\la_2})$ is $D^{\la_1}$ and $\Rad^2(P^{\la_2})/\Soc^2(P^{\la_2})$ is self-dual with composition factors $2[\D{\la_1}] + [\D{\la_2}] + 2[\D{\la_3}] + [\D{\la_4}]$.
Then, by the same argument as above, $\D{\la_i}$ appears in the head of $\Rad^2(P^{\la_2})/\Soc^2(P^{\la_2})$ only when $\D{\la_i}$ appears at least twice as the composition factor, so that $i=1$ or $i=3$, and we may rule out $i=1$ since $\Ext^1(\D{\la_1},\D{\la_1}) = 0$.
Hence the head and the socle of $\Rad^2(P^{\la_2})/\Soc^2(P^{\la_2})$ is $\D{\la_3}$ and $\Rad^3(P^{\la_2})/\Soc^3(P^{\la_2})$ is self-dual.
It follows that 
\[
\Rad^3(P^{\la_2})/\Soc^3(P^{\la_2})\cong M(\beta_1\alpha_1)\oplus \D{\la_4}.
\]
Therefore, relabelling the vertices of the Gabriel quiver as $1,2,\dots,5$ from left to right, we may identify the corresponding indecomposable projective module of the basic algebra of the block as
\begin{gather*}
\bbf e_1 \\
\bbf \beta_1 \\
\bbf \beta_2\beta_1 \\
H \\
\bbf \alpha_3\beta_3\beta_2\beta_1 \\
\bbf \alpha_2\alpha_3\beta_3\beta_2\beta_1 \\
\bbf \alpha_1\alpha_2\alpha_3\beta_3\beta_2\beta_1
\end{gather*}
where $H$ is the direct sum of $\bbf \beta_3\beta_2\beta_1$ and the uniserial module
\begin{gather*}
\bbf \alpha_2\beta_2\beta_1 \\
\bbf \alpha_1\alpha_2\beta_2\beta_1 \\
\bbf \beta_1\alpha_1\alpha_2\beta_2\beta_1.
\end{gather*}
Further, we may obtain several relations, which importantly include $\alpha_1\beta_1 = 0$.

By swapping $\la_1$ and $\la_4$, $\la_2$ and $\la_5$, we obtain the analogous result for $P^{\la_5}$.

In the following, we denote by $P^{\la_i}$ the indecomposable projective modules over the basic algebra by abuse of notation. Then right multiplication by $\alpha_1$ gives the unique homomorphism $P^{\la_2}\to P^{\la_1}$ up to nonzero scalar multiple.
Let $S$ and $T$ be the submodule of $P^{\la_1}$ generated by 
$\beta_2\beta_1\alpha_1$ and $\beta_2$, respectively.
Then, $\Rad(P^{\la_1})$ is the sum of $T$ and the module
\begin{gather*}
\bbf \alpha_1 \\
\bbf \beta_1 \alpha_1 \\
S.
\end{gather*}
Further, $\Rad(S)$ is generated by $\alpha_2\beta_2\beta_1\alpha_1$ and $\beta_3\beta_2\beta_1\alpha_1$, $\Rad(T)$ is generated by 
$\alpha_2\beta_2$ and $\beta_3\beta_2$.
We have a similar result for $P^{\la_4}$.
On the other hand, $\Rad(P^{\la_3})$ is generated by 
$\alpha_2$ and $\beta_3$.

Let $A$ be the quotient algebra of the basic algebra of the block by setting $\alpha_2 = 0$ and $\beta_3 = 0$.
Further, let $t = e_2+e_3+e_4$ and consider $tAt$.
It is easy to check that $tAt\cong \bbf Q/(\varepsilon^2, \varphi^2)$, where $Q$ is the following quiver.
\[
\begin{tikzcd}[scale=2]
	{2} & {3} & {4}
	\arrow[from=1-1, to=1-1, loop left, "\varepsilon" above]
	\arrow[from=1-2, to=1-1, "\beta" above]
	\arrow[from=1-2, to=1-3, "\alpha" above]
	\arrow[from=1-3, to=1-3, loop right, "\phi" above]
\end{tikzcd}
\]
Thus, $tAt$ is a gentle algebra.
Then $tAt$ is Schurian-infinite, by \cite[Proposition~3.3]{Plam19}.
We may also see this by applying \cite[Theorem~1.1]{Plam19}, observing that there is a band $\beta^-\alpha\varphi\alpha^-\beta\varepsilon$, which implies that $tAt$ is representation-infinite.
The result now follows by \cref{reduction}.

\section{Weight 3 blocks}\label{sec:wt3}

In this section, we will prove \cref{thm:main} for weight $3$ blocks.
As in \cref{sec:wt2}, we will heavily rely on \cref{prop:matrixtrick}, and thus, by \cref{thm:jamesconjsmallwt} we may assume that $p=0$, and deduce the result for all $p>3$.

We once again split our blocks into five separate cases, as in \cref{sec:wt2}.
In order to index partitions nicely, we tweak our notation from \cref{subsec:abacus}.
Our notation is analogous to that used for example in \cite{fay08wt3}, but slightly different, to keep it compatible with our weight 2 notation choices.
Recalling that runner $i$ of the abacus display is by definition the runner containing the position $p_i$, in this subsection, $\langle i \rangle$ will denote the partition obtained from the core by sliding the lowest bead on runner $i$ down three places, while $\langle i, j \rangle$ will denote the partition obtained from the core by sliding the lowest bead on runner $i$ down two places, and the lowest bead on runner $j$ down one place.
Note that we allow $i=j$, so that $\langle i, i \rangle$ is obtained by sliding the lowest bead on runner $i$ down two places, and the second lowest bead down one place.
We let $\langle i^2, j \rangle$ denote the partition whose abacus display is obtained from that of the core by sliding down the lowest two beads on runner $i$ and the lowest bead on runner $j$ one place each.
We let $\langle i^3 \rangle$ denote the partition whose abacus display is obtained from that of the core by sliding the lowest three beads on runner $i$ down one place each.
Finally, we let $\langle i, j, k\rangle$, for $i<j<k$, denote the partition whose abacus display is obtained from that of the core by sliding down the lowest bead on runners $i$, $j$, and $k$ down one place each.

In order to also compute the graded decomposition numbers in characteristics 2 and 3, we will use results of Fayers and Tan to determine the adjustment matrices.

\begin{defn}
If $B=B(\rho,3)$ is a weight 3 block of $\hhh$ and the abacus display for $\rho$ satisfies $p_i - p_{i-1} > 2e$ for all $i = 1,\dots, e-1$, then we say that $B$ is a \emph{Rouquier block}.
The collection of all such blocks of weight 3 forms a Scopes equivalence class.
A Rouquier block can be thought of as the maximal Scopes class for each weight.
\end{defn}

\begin{thmc}{faytan06}{Proposition 3.1}
Suppose that $B$ is a weight $3$ Rouquier block of $\hhh$.
\begin{enumerate}[label=(\roman*)]
\item If $\nchar \bbf = 2$, then
$a_{\nu\mu}(v) = 1$ if $(\nu,\mu) = (\langle i^3 \rangle, \langle i \rangle)$ or $(\langle i^2,k \rangle, \langle i,k \rangle)$, for $1\leq i,k \leq e-1$ with $i\neq k$.

\item If $\nchar \bbf = 3$, then
$a_{\nu\mu}(v) = 1$ if $(\nu,\mu) = (\langle i^3 \rangle, \langle i,i \rangle)$ or $(\langle i,i \rangle, \langle i \rangle)$, for $1\leq i \leq e-1$.
\end{enumerate}
For all other $\nu$ and $\mu$, $a_{\nu\mu}(v) = \delta_{\nu\mu}$.
\end{thmc}

For the following result, the reader is invited to see \cite{faytan06} for the exact definition of inducing semi-simply and almost semi-simply -- we will only require the explicit determination of when the former occurs.

\begin{thmc}{faytan06}{Theorem 3.3 and Proposition 3.4}\label{thm:wt3adj}
Suppose that $B$ is a weight $3$ block of $\hhh$.
\begin{enumerate}[label=(\roman*)]
\item If $\nchar \bbf = 2$, then
$a_{\nu\mu}(v) = 1$ if there is a Rouquier block $C$ and some $1\leq i\neq k \leq e-1$ such that
\begin{itemize}
\item $\nu$ induces semi-simply or almost semi-simply to $\langle i^3 \rangle$ in $C$, while $\mu$ induces semi-simply to $\langle i \rangle$ in $C$; or such that
\item $\nu$ induces semi-simply to $\langle i^2, k \rangle$ in $C$, while $\mu$ induces semi-simply to $\langle i, k \rangle$ in $C$.
\end{itemize}

\item If $\nchar \bbf = 3$, then
$a_{\nu\mu}(v) = 1$ if there is a Rouquier block $C$ and some $1\leq i \leq e-1$ such that
\begin{itemize}
\item $\nu$ induces semi-simply to $\langle i^3 \rangle$ in $C$, while $\mu$ induces semi-simply to $\langle i, i \rangle$ in $C$; or such that
\item $\nu$ induces semi-simply to $\langle i, i \rangle$ in $C$, while $\mu$ induces semi-simply to $\langle i \rangle$ in $C$.
\end{itemize}
\end{enumerate}
For all other $\nu$ and $\mu$, $a_{\nu\mu}(v) = \delta_{\nu\mu}$.

In particular, a partition $\la$ in $B$ induces up semi-simply to a partition $\omega$ in $C$ of one of the forms above if and only if $\omega$, $\la$ and $B$ satisfy one of the following sets of conditions, where $1 \leq i < k$.
(Note that as a matter of convention, we consider $p_e$ to be an arbitrary integer satisfying $p_e > p_{e-1} + 2e$.)

\begin{table}
\centering
\begin{tabular}{ccc}\toprule
$\omega$ & $\la$ & Conditions on $B$\\\midrule
& $\langle i \rangle$ & $p_{i+1} - p_i > 2e$\\\cmidrule{2-3}
$\langle i \rangle$ & $\langle i, i+1 \rangle$ & $p_{i+1} - p_i < 2e$, $p_{i+2} - p_i > e$\\\cmidrule{2-3}
& $\langle i, i+1, i+2 \rangle$ & $p_{i+2} - p_i < e$\\\midrule
\multirow{2}{*}{$\langle i, i \rangle$} & $\langle i, i \rangle$ & $p_{i+1} - p_i > e$, $p_i - p_{i-1} > e$\\\cmidrule{2-3}
& $\langle i^2, i+1 \rangle$ & $p_{i+1} - p_i < e$, $p_i - p_{i-1} > e$\\\midrule
$\langle i^3 \rangle$ & $\langle i^3 \rangle$ & $p_i - p_{i-1} > 2e$\\\midrule
& $\langle i, k \rangle$ & $p_k - p_i > 2e$, $p_{i+1} - p_i > e$\\\cmidrule{2-3}
$\langle i, k \rangle$ & $\langle i, i+1, k \rangle$ & $p_k - p_{i+1} > e$, $p_{i+1} - p_i < e$\\\cmidrule{2-3}
& $\langle k^2, i \rangle$ & $p_k - p_{i+1} < e$, $p_k - p_i < 2e$, $p_k - p_{i-1} > e$ \\\midrule
& $\langle k, i \rangle$ & $p_{k+1} - p_k > e$, $p_k - p_i > e$\\\cmidrule{2-3}
\multirow{2}{*}{$\langle k, i \rangle$} & $\langle k, k \rangle$ & $p_{k+1} - p_k > e$, $p_k - p_i < e$, $p_k - p_{i-1} > e$\\\cmidrule{2-3}
& $\langle i, k, k+1 \rangle$ & $p_{k+1} - p_k < e$, $p_k - p_i > e$\\\cmidrule{2-3}
& $\langle k^2, k+1 \rangle$ & $p_{k+1} - p_k < e$, $p_k - p_i < e$, $p_k - p_{i-1} > e$\\\midrule
\multirow{2}{*}{$\langle i^2, k \rangle$} & $\langle i^2, k \rangle$ & $p_k - p_i > e$, $p_i - p_{i-1} > e$\\\cmidrule{2-3}
& $\langle k^3 \rangle$ & $p_k - p_i < e$, $p_k - p_{i-1} > 2e$\\\midrule
\multirow{2}{*}{$\langle k^2, i \rangle$} & $\langle k^2, i \rangle$ & $p_k - p_i > 2e$, $p_k - p_{k-1} > e$\\\cmidrule{2-3}
& $\langle k^3 \rangle$ & $2e > p_k - p_i > e$, $p_k - p_{i-1} > 2e$, $2e > p_k - p_{k-1} > e$\\
\bottomrule
\end{tabular}
\end{table}
\FloatBarrier
\end{thmc}

We will use the above result in the following way.
We will find a set of four partitions in our block such that for each $\mu$ among those four partitions, we have $a_{\nu\mu}(v) = \delta_{\nu\mu}$.
It then follows, e.g.~by the formula above \cref{lem:charfree}, that $d_{\la\mu}^{e,p}(v) = d_{\la\mu}^{e,0}(v)$ for $\la$ and $\mu$ among our four partitions, and thus we may apply \cref{prop:matrixtrick}.
Where possible, we choose our partitions so that $\mu$ does not induce up semi-simply to $\langle i \rangle$ or $\langle i,k \rangle$ if $p=2$, or likewise does not induce up semi-simply to $\langle i, i \rangle$ or $\langle i \rangle$ if $p=3$.
When this is not possible, it can be checked that no partition $\nu$ induces up semi-simply to $\langle i^3 \rangle$ or $\langle i^2, k \rangle$, or to $\langle i^3 \rangle$ or $\langle i, i \rangle$, respectively.

\subsection{$p_{e-1} - p_{e-3} < e$}\label{subsec:wt3firstcase}

We take the four partitions $\langle e-3 \rangle$, $\langle e-1, e-3 \rangle$, $\langle e-2, e-3 \rangle$, and $\langle e-3, e-2 \rangle$.

To compute the corresponding graded decomposition numbers, note that all but runners $e{-}3$, $e{-}2$ and $e{-}1$ may be removed, using \cref{thm:runnerrem}, leaving an abacus display with just three runners.
Since $p_{e-1} - p_{e-3} < e$, the abacus display for the core must be one of the following.
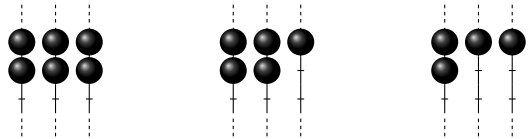
\[
\abacus(vvv,bbb,bbb,nnn,vvv) \qquad\qquad \abacus(vvv,bbb,bbn,nnn,vvv) \qquad\qquad \abacus(vvv,bbb,bnn,nnn,vvv)
\]
They are equivalent, and by convention we work with the first, as in \cref{subsec:wt2firstcase}.
Then our four partitions mentioned above become $(7,1^2)$, $(6,2,1)$, $(5,2^2)$ and $(4,3,2)$, respectively, from which we may deduce that the corresponding submatrix (for $e=3$) is known to be (\ref{targetmatrix}) -- for instance, see \cite[Appendix~B]{mathas}.\\

If $p=2$ or $3$, we may apply \cref{thm:wt3adj} and note that we have chosen our four partitions so that $a_{\nu\mu}(v) = \delta_{\nu\mu}$ for any $\mu$ among them, in both characteristics (since none of these partitions induce up semi-simply to any $\langle i \rangle$, $\langle i, i \rangle$, or $\langle i, k \rangle$). Thus we obtain the same submatrix of the graded decomposition matrix in any characteristic.

\subsection{$p_{e-1} - p_{e-2} < e$ and $p_{e-2} - p_{e-3} < e$, but $p_{e-1} - p_{e-3} > e$}\label{subsec:wt3secondcase}

The four partitions we must examine are $\langle e-2 \rangle$, $\langle e-1, e-2 \rangle$, $\langle e-3 \rangle$, and $\langle e-2, e-3 \rangle$.

We may again remove all but runners $e{-}3$, $e{-}2$ and $e{-}1$, using \cref{thm:runnerrem}, leaving an abacus display with just three runners.
After doing so, the remaining configuration for the $e=3$ abacus display is as below.
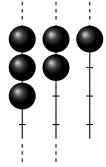
\[
\abacus(vvv,bbb,bbn,bnn,nnn,vv)
\]

Thus, our runner removal applied to the abacus displays for $\langle e-2 \rangle$, $\langle e-1, e-2 \rangle$, $\langle e-3 \rangle$, and $\langle e-2, e-3 \rangle$ yields the four partitions $(8,2)$, $(7,3)$, $(6,2,1^2)$, and $(5,2^2,1)$, respectively.
Now, the corresponding submatrix is known to be (\ref{targetmatrix}) -- e.g.~by \cite[Appendix~B]{mathas} -- which completes this case.\\

If $p=2$ or $3$, we may apply \cref{thm:wt3adj} and note, as in \cref{subsec:wt3firstcase}, that we have chosen our four partitions so that $a_{\nu\mu}(v) = \delta_{\nu\mu}$ for any $\mu$ among them, in both characteristics (since none of these partitions induce up semi-simply to any $\langle i \rangle$, $\langle i, i \rangle$, or $\langle i, k \rangle$). Thus we obtain the same submatrix of the graded decomposition matrix in any characteristic.

\subsection{$p_{e-1} - p_{e-2} > e$ and $p_{e-2} - p_{e-3} < e$}\label{subsec:wt3thirdcase}

Our choice of partitions depends on the values of $p_{e-1}, p_{e-2}$ and $p_{e-3}$, or in other words on which Scopes class our block lies in.
The four partitions we will use are either
\begin{enumerate}[label=(\roman*)]
\item $\langle e-1,e-1 \rangle$, $\langle e-1, e-2 \rangle$, $\langle (e-1)^2, e-2 \rangle$, and $\langle (e-1)^2,e-3 \rangle$;
\item $\langle (e-1)^2, e-2 \rangle$, $\langle e-3, e-1 \rangle$, $\langle e-3 \rangle$, and $\langle e-2, e-3 \rangle$;
or
\item $\langle e-2, e-1 \rangle$, $\langle e-3, e-1 \rangle$, $\langle e-2 \rangle$, and $\langle e-3 \rangle$.
\end{enumerate}
For each case, to compute the corresponding graded decomposition numbers, we may remove all but the $(e{-}3)$-, $(e{-}2)$-, and $(e{-}1)$-runners from the abacus displays, leaving an abacus display for $e=3$.
The remaining core may be any of the cores $\rho^{(i)}$ from \cref{subsec:wt2thirdcase}, and the reader may refer to that section for abacus displays for the first few.
Recall that for $j\geq 1$, we may swap adjacent runners to get between $\rho^{(2j)}$ and either $\rho^{(2j-1)}$ or $\rho^{(2j+1)}$, each involve swapping runners on which the number of beads differs by $j+1$.
Then if $j\geq 2$, this gives us a Scopes equivalence between blocks.
In other words, if we can prove our result for the blocks with cores $\rho^{(1)}=(2)$, $\rho^{(2)}=(3,1)$, and $\rho^{(3)}=(4,2)$, the result will follow.

For the block with core $(2)$ (after reduction by runner removal), we take the four partitions in (i) above, which are $(8,3)$, $(8,2,1)$, $(5,3^2)$, and $(5,3,2,1)$, respectively.
One can check that the corresponding submatrix is (\ref{targetmatrix}), e.g.~by the LLT algorithm~\cite{LLT}.

For the block with core $(3,1)$, we take the four partitions in (ii) above, which are $(6,4,3)$, $(6,3,2,1^2)$, $(5,4,2,1^2)$, and $(4^2,2^2,1)$, respectively.
One can check that the corresponding submatrix is (\ref{targetmatrix3}).

For the block with core $(4,2)$, we take the four partitions in (iii) above, which are $(7,4,3,1)$, $(7,3^2,1^2)$, $(6,5,3,1)$, $(5^2,3,1^2)$, respectively.
Then one can check that the corresponding submatrix is (\ref{targetmatrixalt}), e.g.~by the LLT algorithm.
The result now follows when $p>3$.\\

If $p=2$ or $3$, we may apply \cref{thm:wt3adj} and note, as in \cref{subsec:wt3firstcase}, that we have chosen our four partitions so that $a_{\nu\mu}(v) = \delta_{\nu\mu}$ for any $\mu$ among them, in both characteristics for cores $(2)$ and $(3,1)$, and in characteristic $3$ for the core $(4,2)$.

To see this, note the following.
\begin{enumerate}[label=(\roman*)]
\item In the block with core $(2)$, we have that 
\begin{itemize}
\item $\langle e-1, e-1 \rangle$ induces up semi-simply to $\langle e-1, e-1 \rangle$ in a Rouquier block;

\item $\langle e-1, e-2 \rangle$ induces up semi-simply to $\langle e-1, e-2 \rangle$ in a Rouquier block;

\item $\langle (e-1)^2, e-2 \rangle$ induces up semi-simply to $\langle e-2, e-1 \rangle$ in a Rouquier block;

\item $\langle (e-1)^2, e-3 \rangle$
doesn't induce up semi-simply to any $\langle i \rangle$, $\langle i, i \rangle$, or $\langle i, k \rangle$ in a Rouquier block.
\end{itemize}

Since no $\nu$ induces up semi-simply to $\langle (e-1)^3 \rangle$, $\langle (e-1)^2, e-2 \rangle$, or $\langle (e-2)^2, e-1 \rangle$ in a Rouquier block, we still get that $a_{\nu\mu}(v) = \delta_{\nu\mu}$ for any $\mu$ among our four partitions when $p=2$ or $3$.

\item In the block with core $(3,1)$, we have that none of our partitions induce up semi-simply to any $\langle i \rangle$, $\langle i, i \rangle$, or $\langle i, k \rangle$ except for $\langle (e-1)^2, e-2 \rangle$, which induces up semi-simply to $\langle e-2, e-1 \rangle$ in a Rouquier block.
Since no $\nu$ induces up semi-simply to $\langle (e-2)^2, e-1 \rangle$ in a Rouquier block, we still get that $a_{\nu\mu}(v) = \delta_{\nu\mu}$ for any $\mu$ among our four partitions when $p=2$ or $3$.

\item In the block with core $(4,2)$, none of our partitions induce up semi-simply to any $\langle i \rangle$ or $\langle i, i \rangle$, except for $\langle e-2 \rangle$, which induces up semi-simply to $\langle e-2 \rangle$ in a Rouquier block.
Since no $\nu$ induces up semi-simply to $\langle e-2, e-2 \rangle$ in a Rouquier block, we still get that $a_{\nu\mu}(v) = \delta_{\nu\mu}$ for any $\mu$ among our four partitions when $p = 3$.
\end{enumerate}

Thus we obtain the same submatrix of the graded decomposition matrix in any characteristic, except for the block with core $(4,2)$ in characteristic $2$.
For this, we must choose different partitions -- we take $\langle e-1, e-1 \rangle$, $\langle (e-1)^2, e-2 \rangle$, $\langle e-2, e-1 \rangle$, and $\langle e-3, e-1 \rangle$.
We may again use runner-removal to reduce the computation of the characteristic zero graded decomposition numbers, yielding partitions $(10,5)$, $(7,5,2,1)$, $(7,4,3,1)$, and $(7,3^2,1^2)$, whence we obtain the submatrix (\ref{targetmatrix3}).

Now, since none of our four partitions induce semi-simply to any $\langle i \rangle$ or $\langle i, k \rangle$, by \cref{thm:wt3adj}, 
we have $a_{\nu\mu}(v) = \delta_{\nu\mu}$ for any $\mu$ among them, and thus that we obtain the same submatrix of the graded decomposition matrix in characteristics $0$ and $2$.

Thus, in all cases above, \cref{prop:matrixtrick} yields our result.

\subsection{$p_{e-1} - p_{e-2} < e$ and $p_{e-2} - p_{e-3} > e$}\label{subsec:wt3fourthcase}

Our choice of partitions depends on the values of $p_{e-1}, p_{e-2}$ and $p_{e-3}$, or in other words on which Scopes class our block lies in.
The four partitions we will use are either
\begin{enumerate}[label=(\roman*)]
\item $\langle e-1,e-1 \rangle$, $\langle e-2, e-2 \rangle$, $\langle (e-1)^2, e-2 \rangle$, and $\langle (e-2)^2,e-1 \rangle$;

\item $\langle e-2 \rangle$, $\langle e-1, e-2 \rangle$, $\langle e-1, e-1 \rangle$, and $\langle e-2, e-2 \rangle$;
or

\item $\langle e-1 \rangle$, $\langle e-2 \rangle$, $\langle e-1, e-2 \rangle$, and $\langle e-1, e-1 \rangle$.
\end{enumerate}
For each case, to compute the corresponding graded decomposition numbers, we may remove all but the $(e{-}3)$-, $(e{-}2)$-, and $(e{-}1)$-runners from the abacus displays, leaving an abacus display for $e=3$.

The remaining core may be any of the cores $\sigma^{(i)}$ from \cref{subsec:wt2fourthcase}, and the reader may refer to that section for abacus displays for the first few.
As in \cref{subsec:wt3thirdcase}, our Scopes equivalences reduce the problem down to solving for the blocks with cores 
$\sigma^{(1)} = (1^2)$, $\sigma^{(2)} = (2,1^2)$, and $\sigma^{(3)} = (2^2,1^2)$.

For the block with core $(1^2)$ (after reduction by runner removal), we take the four partitions in (i) above, which are $(7,2^2)$, $(6,2^2,1)$, $(4^2,3)$, and $(4^2,2,1)$, respectively.
One can check that the corresponding submatrix is (\ref{targetmatrixalt}), e.g.~by the LLT algorithm.

For the block with core $(2,1^2)$, we take the four partitions in (ii) above, which are $(9,3,1)$, $(8,4,1)$, $(8,3,2)$, and $(6,3,2^2)$, respectively.
One can check that the corresponding submatrix is (\ref{targetmatrix3}).

For the block with core $(2^2,1^2)$, we take the four partitions in (iii) above, which are $(11,2,1^2)$, $(10,3,1^2)$, $(8,5,1^2)$, and $(8,3^2,1)$, respectively.
One can check that the corresponding submatrix is (\ref{targetmatrix}).
The result now follows when $p>3$.\\

If $p=2$ or $3$, we may apply \cref{thm:wt3adj} and note, as in \cref{subsec:wt3firstcase}, that we have chosen our four partitions so that $a_{\nu\mu}(v) = \delta_{\nu\mu}$ for any $\mu$ among them, in both characteristics for cores $(1^2)$ and $(2,1^2)$, and in characteristic $3$ for the core $(2^2,1^2)$.

To see this, note the following.
\begin{enumerate}[label=(\roman*)]
\item In the block with core $(1^2)$, we have that 
\begin{itemize}
\item $\langle e-1, e-1 \rangle$ induces up semi-simply to $\langle e-1, e-2 \rangle$ in a Rouquier block $C$;

\item $\langle e-2, e-2 \rangle$ doesn't induce up semi-simply to any $\langle i \rangle$, $\langle i, i \rangle$, or $\langle i, k \rangle$ in a Rouquier block $C$;

\item $\langle (e-1)^2, e-2 \rangle$ induces up semi-simply to $\langle e-2, e-1 \rangle$ in a Rouquier block $C$;

\item $\langle (e-2)^2, e-1 \rangle$ induces up semi-simply to $\langle e-2, e-2 \rangle$ in a Rouquier block $C$.
\end{itemize}

Since no $\nu$ induces up semi-simply to $\langle (e-1)^2, e-2 \rangle$, $\langle (e-2)^2, e-1 \rangle$, or $\langle (e-2)^3 \rangle$ in $C$, we still get that $a_{\nu\mu}(v) = \delta_{\nu\mu}$ for any $\mu$ among our four partitions when $p=2$ or $3$.

\item In the block with core $(2,1^2)$, we have that none of our partitions induce up semi-simply to any $\langle i \rangle$, $\langle i, i \rangle$, or $\langle i, k \rangle$ except for $\langle e-1, e-1 \rangle$, which induces up semi-simply to $\langle e-1, e-2 \rangle$ in a Rouquier block $C$.
Since no $\nu$ induces up semi-simply to $\langle (e-1)^2, e-2 \rangle$ in a Rouquier block $C$, we still get that $a_{\nu\mu}(v) = \delta_{\nu\mu}$ for any $\mu$ among our four partitions when $p=2$ or $3$.

\item In the block with core $(2^2,1^2)$, none of our partitions induce up semi-simply to any $\langle i \rangle$ or $\langle i, i \rangle$, except for $\langle e-1 \rangle$, which induces up semi-simply to $\langle e-1 \rangle$ in a Rouquier block.
Since no $\nu$ induces up semi-simply to $\langle e-1, e-1 \rangle$ in a Rouquier block, we still get that $a_{\nu\mu}(v) = \delta_{\nu\mu}$ for any $\mu$ among our four partitions when $p = 3$.
\end{enumerate}

Thus we obtain the same submatrix of the graded decomposition matrix in any characteristic, except for the block with core $(2^2,1^2)$ in characteristic $2$.
For this, we must choose different partitions -- we take $\langle e-2 \rangle$, $\langle e-1, e-2 \rangle$, $\langle e-1, e-1 \rangle$, and $\langle e-2, e-2 \rangle$.
We may again use runner-removal to reduce the computation of the characteristic zero graded decomposition numbers, yielding partitions $(10,3,1^2)$, $(8,5,1^2)$, $(8,3^2,1)$, and $(7,3^2,2)$, whence we obtain the submatrix (\ref{targetmatrix3}).

Now, by \cref{thm:wt3adj}, none of our partitions induce semi-simply to any $\langle i \rangle$ or $\langle i, k \rangle$, except for $\langle e-1, e-1 \rangle$, which induces semi-simply to $\langle e-1, e-2 \rangle$ in a Rouquier block.
But since no $\nu$ induces semi-simply to $\langle (e-1)^2, e-2 \rangle$,
we have $a_{\nu\mu}(v) = \delta_{\nu\mu}$ for any $\mu$ among our four partitions, and thus we obtain the same submatrix of the graded decomposition matrix in characteristics $0$ and $2$.

Thus, in all cases above, \cref{prop:matrixtrick} yields our result.

\subsection{$p_{e-1} - p_{e-2} > e$ and $p_{e-2} - p_{e-3} > e$}\label{subsec:wt3fifthcase}

Across various cases, we will use partitions $\langle e-1,e-1 \rangle, \langle e-1,e-2 \rangle, \langle (e-1)^2,e-2 \rangle, \langle e-2,e-1 \rangle, \langle (e-1)^3 \rangle, \langle (e-2)^2,e-1 \rangle, \langle e-2, e-2 \rangle$.
Notice that in all cases, we are able to apply \cref{thm:runnerrem}, reducing our computations down to the $e=3$ situation.

Recall that in \cref{subsec:wt2fifthcase}, we were able to  index the cores that arise in this case by the pair of integers $(ke+1,je+1)$ or $(ke+2, je+2)$, where $k, j \in \bbz_{>0}$, and denote these cores by $\kappa_{(ke+1,je+1)}$ or $\kappa_{(ke+2,je+2)}$.
As stated in \cref{subsec:wt2fifthcase}, we may pass between these cores by runner swaps, and these swaps involve swapping runners on which the number of beads differs by either
$k+j+1$ (for passing between $\kappa_{(ke+1,je+1)}$ and $\kappa_{(ke+2,je+2)}$),
$k$ (for passing between $\kappa_{((k-1)e+2,je+2)}$ and $\kappa_{(ke+1,je+1)}$),
or $j$ (for passing between $\kappa_{(ke+2,(j-1)e+2)}$ and $\kappa_{(ke+1,je+1)}$).
Since we are now in weight $3$, these are Scopes equivalences except for small $k$ and $j$, and we can see that there are four Scopes classes in total to check, with (minimal) representatives
$\kappa_{(e+1,e+1)} = (3,1^2)$, $\kappa_{(2e+1,e+1)} = (5,3,1^2)$, $\kappa_{(e+1,2e+1)} = (4,2^2,1^2)$, and $\kappa_{(2e+1,2e+1)} = (6,4,2^2,1^2)$.

For the block with core $(3,1^2)$, we may take the four partitions $\langle e-1, e-1 \rangle = (9,4,1)$, $\langle e-1,e-2 \rangle = (9,3,2)$, $\langle (e-1)^2,e-2 \rangle = (6,4^2)$, and $\langle (e-1)^3 \rangle = (6,4,2^2)$, respectively.
One can check that the corresponding submatrix is (\ref{targetmatrix}), e.g.~by the LLT algorithm.

For the block with core $(5,3,1^2)$, we may take the four partitions $\langle (e-1)^2,e-2 \rangle = (8,6,3,2)$, $\langle e-2,e-1 \rangle = (8,5,4,2)$, $\langle (e-2)^2,e-1 \rangle = (8,3^2,2^2,1)$, and $\langle e-2, e-2 \rangle = (5^2,4,2^2,1)$, respectively.
One can check that the corresponding submatrix is (\ref{targetmatrixalt}), e.g.~by the LLT algorithm.

For the block with core $(4,2^2,1^2)$, we may take the four partitions $\langle e-1,e-2 \rangle = (10,4,3,1^2)$, $\langle (e-1)^2,e-2 \rangle = (7,5^2,1^2)$, $\langle (e-1)^3 \rangle = (7,5,3^2,1)$, and $\langle (e-2)^2,e-1 \rangle = (7,4,3^2,2)$, respectively.
One can check that the corresponding submatrix is (\ref{targetmatrix2}).

For the block with core $(6,4,2^2,1^2)$, we may take the four partitions $\langle e-1,e-2 \rangle = (12,4^2,3,1^2)$, $\langle (e-1)^2,e-2 \rangle = (9,7,4,3,1^2)$, $\langle e-2,e-1 \rangle = (9,6,5,3,1^2)$, and $\langle (e-2)^2,e-1 \rangle = (9,4^2,3^2,2)$, respectively.
One can check that the corresponding submatrix is (\ref{targetmatrixaltsquare}).
The result now follows if $p>3$.

If $p=3$, we may apply \cref{thm:wt3adj} and note, as in \cref{subsec:wt3firstcase}, that we have chosen our four partitions so that $a_{\nu\mu}(v) = \delta_{\nu\mu}$ for any $\mu$ among them.
Checking this is similar to previous cases, and we leave the details to the reader.

Now suppose that $p=2$.

For the block with core $(3,1^2)$, we may take the four partitions $\langle e-1, e-2 \rangle = (9,3,2)$, $\langle e-2 \rangle = (8,4,2)$, $\langle e-2, e-1 \rangle = (6^2,2)$, and $\langle (e-1)^2, e-2 \rangle = (6,4^2)$, respectively.
One may check as above that the corresponding submatrix is (\ref{targetmatrix}) in characteristics $0$ and $2$.

For the block with core $(4,2^2,1^2)$, we may take the four partitions $\langle e-1, e-2 \rangle = (10,4,3,1^2)$, $\langle e-2 \rangle = (9,5,3,1^2)$, $\langle e-2, e-1 \rangle = (7^2,3,1^2)$, and $\langle (e-1)^2, e-2 \rangle = (7,5^2,1^2)$, respectively.
One may check as above that the corresponding submatrix is (\ref{targetmatrix}) in characteristics $0$ and $2$.

For the block with core $(5,3,1^2)$, it suffices to note that the $\#$-automorphism on $\hhh$ (e.g.~see \cite[Exercise~3.14]{mathas}) induces an isomorphism between this block and the one with core $(5,3,1^2)' = (4,2^2,1^2)$.

If $p=2$ the block with core $(6,4,2^2,1^2)$ (i.e.~the Rouquier block) is not susceptible to our methods when $e=3$, and we must treat it separately.

For now we will assume that $e\geq 4$.
Under this assumption, we will take either the partition $\langle (e-1)^3 \rangle$ or the partition $\langle (e-1)^2, e-2 \rangle$ as our most dominant partition, along with the three partitions $\langle (e-1)^2, e-3 \rangle$, $\langle (e-2)^2, e-1 \rangle$, and $\langle e-3, e-2, e-1 \rangle$.
Then, for any pair of these partitions, we may apply \cref{thm:runnerrem} to remove all but four runners when computing the characteristic zero graded decomposition numbers.
One may easily check, using the system of runner swaps at the end of \cref{subsec:wt2fifthcase}, that up to Scopes equivalence, these decomposition numbers boil down to those of just 10 blocks, with the following cores (in the notation of \cref{subsec:wt2fifthcase}).
\begin{multicols}{4}
\begin{enumerate}[label=(\roman*)]
\item $\kappa_{e+1, e+1, 1}$,

\item $\kappa_{e+1, e+2, 3}$,

\item $\kappa_{e+1, e+1, e+1}$,

\item $\kappa_{e+1, e+1, 2e+1}$,

\item $\kappa_{e+1, 2e+1, e+1}$,

\item $\kappa_{e+1, 2e+1, 2e+1}$,

\item $\kappa_{2e+1, e+1, e+1}$,

\item $\kappa_{2e+1, e+1, 2e+1}$,

\item $\kappa_{2e+1, 2e+1, e+1}$,

\item $\kappa_{2e+1, 2e+1, 2e+1}$.
\end{enumerate}
\end{multicols}

Since none of the cores in cases (i)--(viii) satisfy both $p_{e-1} - p_{e-2}>2e$ and $p_{e-2} - p_{e-3}>2e$, we may remove an extra runner and handle them as above.
However, the cores (ix) and (x) each have $p_{e-1} - p_{e-2}=2e+1$ and $p_{e-2} - p_{e-3}=2e+1$, and would thus reduce to the Rouquier block with core $(6,4,2^2,1^2)$ if we attempted to strip off an extra runner.

One may check that if we choose the partitions  $\langle (e-1)^2, e-2 \rangle$, $\langle (e-1)^2, e-3 \rangle$, $\langle (e-2)^2, e-1 \rangle$, and $\langle e-3, e-2, e-1 \rangle$, it yields (\ref{targetmatrixalt}) as the corresponding submatrix of the graded decomposition matrix in each case.

Concretely, in each case the corresponding partitions are as follows.

\begin{enumerate}[label=(\roman*)]
\item[(ix)] $(15,12,8,6,3^2,1^3)$, $(15,12,5^2,3^3,2^2)$, $(15,8^2,6^2,4,1^3)$, $(15,8^2,6,3^3,2^2)$;

\item[(x)] $(16,13,9,7,4^2,2^3,1^3)$, $(16,13,6^2,4^3,3^2,1^3)$, $(16,9^2,7^2,5,2^3,1^3)$, $(16,9^2,7,4^3,3^2,1^3)$.
\end{enumerate}

Now, for any $e$, one may check that none of $\langle (e-1)^2, e-3 \rangle$, $\langle (e-2)^2, e-1 \rangle$, or $\langle e-3, e-2, e-1 \rangle$ induce up semi-simply to any $\langle i \rangle$ or $\langle i, k \rangle$ in either of these cases (i.e.~whenever $p_{e-1} - p_{e-2} > e$ and $p_{e-2} - p_{e-3} > e$), 
while $\langle (e-1)^2, e-2 \rangle$ doesn't induce up semi-simply to any $\langle i \rangle$ or $\langle i, k \rangle$ in either case (i.e.~whenever $p_{e-1} - p_{e-2}, p_{e-2} - p_{e-3} > 2e$).
Thus $a_{\nu \mu} = 0$ whenever $\mu$ is any one of our chosen partitions, and the result follows as in the previous sections.

It is worth noting that even though our chosen partitions involve sliding beads on only three runners, if we use runner-removal to reduce to abacus displays with three runners, our partitions are not all $3$-regular.

We now return to the Rouquier block when $e=3$ and $p=2$, and will directly show that this block is Schurian-infinite.

Note that we may use the LLT algorithm to compute the characteristic $0$ decomposition matrices, and then apply \cref{thm:wt3adj} (along with \cite[Proposition~3.5]{faytan06} in order to determine when a partition induces almost semi-simply to another) to determine those in characteristic $2$.
In the block with core $(6,4,2^2,1^2)$, the five most dominant $3$-regular partitions are $\la_1 := (15,4,2^2,1^2)$, $\la_2 := (12,7,2^2,1^2)$, $\la_3 := (12,4^2,3,1^2)$, $\la_4 := (9,7,5,2,1^2)$, and $\la_5 := (9,7,4,3,1^2)$.
We see that the characteristic $2$ decomposition matrix starts with the following 5 rows.
\[
\begin{array}{c|ccccc}
 & \la_1 & \la_2 & \la_3 & \la_4 & \la_5 \\\hline
\la_1 & 1 & \cdot & \cdot & \cdot & \cdot \\
\la_2 & 0 & 1 & \cdot & \cdot & \cdot \\
\la_3 & 1 & 1 & 1 & \cdot & \cdot \\
\la_4 & 1 & 0 & 0 & 1 & \cdot \\
\la_5 & 1 & 1 & 1 & 1 & 1 \\
\end{array}
\]
It follows that we have
\[
\spe{\la_1} = \D{\la_1}, \quad \spe{\la_2} = \D{\la_2}, \quad \Rad(\spe{\la_4}) \cong \D{\la_1}.
\]
We may compute that there are homomorphisms from $\spe{\la_1}$ and $\spe{\la_2}$ to $\spe{\la_3}$, so that $\Rad(\spe{\la_3}) \cong \D{\la_1} \oplus \D{\la_2}$.
For example, working with graded (column) Specht modules for KLR algebras \cite[Section~7]{kmr}, one may check that these homomorphisms are given by 
\[
z^{\la_1} \longmapsto (\psi_{16} \psi_{15} \psi_{14} \psi_{13} \psi_{17} \psi_{16} \psi_{18} \psi_{17} \dots \psi_{14} + \psi_{22} \psi_{21} \dots \psi_{13} \psi_{23} \psi_{22} \dots \psi_{16} \psi_{24} \psi_{23} \dots \psi_{14}) z^{\la_3}
\]
and
\[
z^{\la_2} \longmapsto \psi_{15} \psi_{14} \psi_{13} \psi_{17} \psi_{16} \psi_{19} \psi_{18} \psi_{17} \psi_{16} \psi_{15} \psi_{14} z^{\la_3}.
\]
It is easy to check that these define nonzero elements in $\spe{\la_3}$ -- replacing each generator $\psi_i$ with the corresponding basic transposition $s_i$ and acting naturally on the (column) initial $\la_3$-tableau yields a standard $\la_3$-tableau, and indeed each expression is equal to the element of the standard homogenous basis (c.f.~\cite[Corollary~7.20]{kmr}) indexed by that tableau, up to some lower order terms.

We may also compute that there are homomorphisms from $\spe{\la_1}$ and $\spe{\la_2}$ to $\spe{\la_5}$, so that $\D{\la_1} \oplus \D{\la_2}$ is a submodule of $\spe{\la_5}$, and that there are homomorphisms from $\spe{\la_3}$ and $\spe{\la_4}$ to $\spe{\la_5}$.
These are given by
\[
z^{\la_1} \longmapsto \psi_{19} \psi_{18} \psi_{17} \psi_{16} \psi_{20} \psi_{19} \psi_{18} \psi_{21} \psi_{20} \psi_{22} \psi_{21} \dots \psi_{13} \psi_{23} \psi_{22} \dots \psi_{16} \psi_{24} \psi_{23} \dots \psi_{14} z^{\la_5},
\]
\[
z^{\la_2} \longmapsto \psi_{15} \psi_{14} \psi_{13} \psi_{17} \psi_{16} \psi_{19} \psi_{18} \dots \psi_{14} \psi_{22} \psi_{21} \psi_{20} \psi_{19} \psi_{23} \psi_{22} \psi_{21} \psi_{24} \psi_{23} z^{\la_5},
\]
\[
z^{\la_3} \longmapsto \psi_{22} \psi_{21} \psi_{20} \psi_{19} \psi_{23} \psi_{22} \psi_{21} \psi_{24} \psi_{23} z^{\la_5},
\]
and
\[
z^{\la_4} \longmapsto \psi_{18} \psi_{17} \psi_{16} \psi_{15} \psi_{14} z^{\la_5}.
\]
One may check that composing homomorphisms from $\spe{\la_1}$ and $\spe{\la_2}$ to $\spe{\la_3}$ and from $\spe{\la_1}$ to $\spe{\la_4}$ with these latter homomorphisms gives nonzero homomorphisms, so that $\Soc(\spe{\la_5}) \cong \D{\la_1} \oplus \D{\la_2}$ and $\spe{\la_5}$ has heart $\D{\la_3} \oplus \D{\la_4}$.

It follows that we have extensions yielding $A^{(1)}_3$ with zigzag orientation as a subquiver of the Gabriel quiver, on vertices $\la_1$, $\la_3$, $\la_4$, and $\la_5$.
The result now follows by \cref{cor:Schurianinfinitequiver}.

\section{Blocks of weight at least 4}\label{sec:highwt}

Since we will deal with blocks of arbitrarily large weight, we must employ some new conventions for partitions and their abacus displays.

We will use the beta-numbers for an abacus display with $r$ beads, as in \cref{subsec:scopes}, and recall that we have fixed the integers $p_0 < p_1 < \dots < p_{e-1}$, determined by the core $\rho$, when looking at partitions in a block $B(\rho,w)$, as before.
We will use an $e$-quotient notation for our partitions, adapted so that it follows this ordering on $p_i$.
Ordering the runners starting with the runner containing the $p_{e-1}$ position, and then the runner containing the $p_{e-2}$ position, and so on until the $p_0$ runner, we may read a partition from each runner, considered as a $1$-runner abacus display.
We will use the shorthand notation $\varnothing^k$ to denote a string of $k$ components, each equal to the empty partition, in the $e$-quotient notation for a partition $\la$.

\begin{eg}
Let $e=4$, $\rho = (3,1^2)$, $\la = (8^2,2^2,1)$, and take $r=7$.
Then $\rho$ has beta numbers $(9,6,5,3,2,1,0)$, while $\la$ has beta numbers $(14,13,6,5,3,1,0)$, which are displayed on an abacus as below.
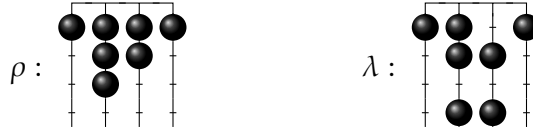
\[
\rho: \; \abacus(lmmr,bbbb,nbbn,nbnn,nnnn)
\qquad\qquad\qquad
\la: \; \abacus(lmmr,bbnb,nbbn,nnnn,nbbn)
\]
The partition $\la$ has weight $4$.
In the abacus display for its core, $\rho$, we can read off the integers $p_0 = 0$, $p_1 = 3$, $p_2 = 6$, and $p_3 = 9$.
In $e$-quotient notation, we write $\la = ((1),(2,1), \varnothing^2)$.
\end{eg}

While the $e$-quotient is usually only well-defined up to some cyclic permutation, our variant that orders the runners in terms of the integers $p_i$ is unique.
In other words, the $e$-quotient notation for a partition $\la$ does not change if we choose a different $r$, though the abacus display itself does.

It was shown in \cref{sec:wt2} that weight $2$ blocks are Schurian-infinite, by breaking into five cases depending on the the comparative values of $p_{e-1}$, $p_{e-2}$, and $p_{e-3}$ for the block's core $\rho$.
When $p_{e-1} - p_{e-2} >e$, handled in \cref{subsec:wt2thirdcase,subsec:wt2fifthcase}, the methods in characteristic $2$ did not directly appeal to \cref{prop:matrixtrick}, as no such submatrices are available in such cases.
Instead, results of Fayers about extensions between simples \cite{faywt2} were used directly, and when $p_{e-2} - p_{e-3} > e$ and $e=3$, separate ad hoc methods were required.

We will initially avoid the case $e=3$, $p=2$, and thus have a bit more flexibility to work with.
The following result essentially deduces the same result as \cref{subsec:wt2thirdcase,subsec:wt2fifthcase} -- that $B(\rho,2)$ is Schurian-infinite when $p_{e-1} - p_{e-2} >e$ -- using \cref{prop:matrixtrick} directly, under the additional assumption that we're not in this difficult $e=3$, $p=2$ case.
This lemma will be useful to quickly prove that all blocks of weight at least $4$ are Schurian-infinite, in conjunction with those submatrices already determined in \cref{sec:wt2}.

\begin{lem}\label{lem:wt2rouquish}
Suppose $e\geq 4$ and $\rho$ is a core satisfying $p_{e-1} - p_{e-2} >e$.
Define four partitions as follows.
\[
\la^{(1)} = ((1),(1),\varnothing^{e-2}), \qquad \la^{(2)} = ((1),\varnothing,(1),\varnothing^{e-3}).
\]
\begin{enumerate}[label=(\roman*)]
\item
If $p_{e-2} - p_{e-3} <e$, then define 
\[
\la^{(3)} = (\varnothing,(2),\varnothing^{e-2}), \qquad \la^{(4)} = (\varnothing,\varnothing,(2),\varnothing^{e-3}).
\]

\item
If $p_{e-2} - p_{e-3} >e$, then define 
\[
\la^{(3)} = (\varnothing,(1^2),\varnothing^{e-2}), \qquad \la^{(4)} = (\varnothing,(1),(1),\varnothing^{e-3}).
\]
\end{enumerate}
Then the four partitions $\la^{(1)}$, $\la^{(2)}$, $\la^{(3)}$, and $\la^{(4)}$ give (\ref{targetmatrixalt}) as a submatrix of the graded decomposition matrix in any characteristic.
It follows from \cref{prop:matrixtrick} that $B(\rho,2)$ is Schurian-infinite.
\end{lem}

\begin{proof}
We essentially argue as in the proofs in \cref{sec:wt2}.
We may first apply \cref{thm:runnerrem} $e-4$ times to determine $d_{\la\mu}^{e,0}$ by examining the leading four runners alone.
One may check that up to Scopes equivalence, the remaining partitions live in one of five blocks, with cores $(3)$, $(4,1^2)$, and $(4,1^3)$ satisfying $p_{e-2} - p_{e-3} <e$, and cores $(5,2^2)$ and $(6,3^2,1^3)$ satisfying $p_{e-2} - p_{e-3} >e$.
Since we are in a block of weight 2, the decomposition matrix in characteristic $p$ for $p \geq 3$ is identical to the decomposition matrix in characteristic $0$, by~\cref{thm:jamesconj}.
Hence in each case, we may easily verify, for example by the LLT algorithm, that the corresponding submatrices are each (\ref{targetmatrixalt}) if $p\neq 2$.
By applying \cref{thm:wt2adjust}, we may check that these submatrices are in fact identical when $p=2$.
\end{proof}

\subsection{When $e\neq 3$ or $p \neq 2$}

\begin{thm}\label{thm:mostblocks}
Let $e\geq 3$, and $w\geq 4$, and let $\rho$ be an $e$-core.
Unless $e=3$ and $p = 2$, the weight $w$ block $B(\rho,w)$ of $\hhh$ is Schurian-infinite.
\end{thm}

\begin{proof}
In our proof, we consider 11 different cases, two of which then split into two further cases. We describe these cases, together with our method of dealing with each of them, in Table \ref{tab1}. 

Our method of proof will be to reduce to weight 2 and then apply the results of \cref{sec:wt2}.
In order to do this, we will slide the lowest bead on the abacus display of $\rho$ down $w-2$ spaces on each of four (or five) partitions, and then use the row-removal results of \cref{thm:rowremFock,thm:rowremDecomp} to remove this bead (i.e.~the first row of each of the four partitions).
To then apply results of \cref{sec:wt2}, we must know the position of the bead above the lowest bead on the abacus display for $\rho$, in order to know the new ordering on runners.
If $p_{e-1}-p_{e-2} > e$, then the ordering on runners is unchanged -- this is reflected in our first seven cases below, five of which are essentially identical to the five subsections of \cref{sec:wt2} once we've removed the first row, while the other two are handled by \cref{lem:wt2rouquish}.

Thus we will list the four partitions required in each case below.
The proof then follows by applying the known decomposition numbers in weight 2, where our chosen partitions were shown in \cref{sec:wt2} and \cref{lem:wt2rouquish} to satisfy the conditions of \cref{prop:matrixtrick}.

We summarise our results in Table \ref{tab1}, where the following shorthand is used: 
\begin{align*}
&\alpha \text{ denotes } p_i -p_j <e, && \beta \text{ denotes } e<p_i -p_j <2e, \\
 &\gamma \text{ denotes } 2e < p_i-p_j,&& \delta \text{ denotes } e< p_i -p_j.\qedhere
\end{align*}

\afterpage{%
    \clearpage
    \thispagestyle{empty}
\begin{landscape}
\begin{center}
\renewcommand{\arraystretch}{1.2}
\[\begin{array}{|m{10pt}m{10pt}m{10pt}|m{15pt}m{15pt}|m{10pt}m{10pt}m{10pt}|m{15pt}m{15pt}|ll|ll|l} \hline
\multicolumn{3}{|c|}{p_{e-1}-p_{e-2}} & \multicolumn{2}{c|}{p_{e-2}-p_{e-3}} & 
\multicolumn{3}{c|}{p_{e-1}-p_{e-3}} & \multicolumn{2}{c|}{p_{e-1}-p_{e-4}} & \multicolumn{4}{c|}{} \\ 
$\alpha$ &$\beta$&$\gamma$ &\multicolumn{1}{c}{\alpha} &\multicolumn{1}{c|}{\delta} & $\alpha$ &  $\beta$& $\gamma$ & \multicolumn{1}{c}{\alpha} & \multicolumn{1}{c|}{\delta} & \multicolumn{4}{c|}{} \\ \hline 
&$\bullet$&&\multicolumn{1}{c}{\bullet}&&&$\bullet$&&&&((w-2,2), \vn^{e-1}) & ((w-2),(2),\vn^{e-2}) & \text{(\ref{targetmatrix})} &\text{\cref{subsec:wt2firstcase}} \\
&&&&&&&&&&((w-2),\vn,(2),\vn^{e-3}) & ((w-2,1),\vn,(1),\vn^{e-3}) && \\ \hline 
&$\bullet$&&\multicolumn{1}{c}{\bullet} &&&&$\bullet$&&&((w-2,2), \vn^{e-1}) & ((w-2),(2),\vn^{e-2}) & \text{(\ref{targetmatrix})} &\text{\cref{subsec:wt2secondcase}}\\
&&&&&&&&&&((w-2,1),(1),\vn^{e-2}) & ((w-2,1^2),\vn^{e-1}) &&\\ \hline 
&&$\bullet$&\multicolumn{1}{c}{\bullet}&&&&&& &((w-2,2), \vn^{e-1}) & ((w-2,1),(1),\vn^{e-2}) & \text{(\ref{targetmatrix})} &\text{\cref{subsec:wt2thirdcase}} \hspace{12pt} p\neq2\\
&&&&&&&&&&((w-2),(2),\vn^{e-2}) & ((w-2),\vn,(2),\vn^{e-3}) &  & \\ \hline 
&&$\bullet$&\multicolumn{1}{c}{\bullet}&&&&&& &((w-2,1),(1), \vn^{e-2}) & ((w-2,1),\vn,(1),\vn^{e-3}) &\text{(\ref{targetmatrixalt})} & \text{\cref{lem:wt2rouquish}} \hspace{28pt} p=2  \\
&&&&&&&&&&((w-2),(2),\vn^{e-2}) & ((w-2),\vn,(2),\vn^{e-3}) & &\hspace{82pt} e \geq 4 \\ \hline 
&$\bullet$ & & &$\bullet$ &&&&&& ((w-2,2),\vn^{e-1}) & ((w-2),(2),\varnothing^{e-2}) & \text{(\ref{targetmatrix})} &\text{\cref{subsec:wt2fourthcase}}\\
&&&&&&&&&&((w-2,1),(1),\vn^{e-2}) & ((w-2,1^2),\vn^{e-1}) && \\ \hline
&&$\bullet$ &&$\bullet$ &&&&&& ((w-2,2),\vn^{(e-1)}) & ((w-2,1^2),\vn^{e-1}) &  \text{(\ref{targetmatrixstar})} &\text{\cref{subsec:wt2fifthcase}} \hspace{12pt} p\neq 2\\
&&&&&&&&&&((w-2,1),(1),\vn^{e-2}) & ((w-2),(2),\vn^{e-2}) & &\\
&&&&&&&&&&((w-2),(1^2),\vn^{e-2}) &&&  \\ \hline 
&&$\bullet$ &&$\bullet$ &&&&&& ((w-2,1),(1),\vn^{(e-2)}) & ((w-2,1),\vn,(1),\vn^{e-3}) &  \text{(\ref{targetmatrixalt})} &\text{\cref{lem:wt2rouquish}} \hspace{27pt} p =2\\
&&&&&&&&&&((w-2),(1^2),\vn^{e-2}) & ((w-2),(1),(1),\vn^{e-3}) & &\hspace{81pt} e \geq 4 \\ \hline
&&&&&&&&\multicolumn{1}{c}{\bullet} && ((w-2),(2),\vn^{e-2}) & ((w-2),\vn,(2),\vn^{e-3}) & \text{(\ref{targetmatrix})} &\text{\cref{subsec:wt2firstcase}} \\
&&&&&&&&&&((w-2),\vn^2,(2),\vn^{e-4}) & ((w-2),(1),\vn,(1),\vn^{e-4}) && \\ \hline
&&&&& $\bullet$&&&& \multicolumn{1}{c|}{\bullet} & ((w-2),(2),\vn^{e-2}) & ((w-2),\vn,(2),\vn^{e-3}) &  \text{(\ref{targetmatrix})} &\text{\cref{subsec:wt2firstcase}}\\
&&&&&&&&&&((w-2,1),\vn^{e-1}) & ((w-2,1),(1),\vn^{e-2}) && \\ \hline 
$\bullet$ &&&\multicolumn{1}{c}{\bullet} &&& $\bullet$ &&&& ((w-2),(2),\vn^{e-2}) &((w-2,2),\vn^{e-1}) &  \text{(\ref{targetmatrix})} &\text{\cref{subsec:wt2firstcase}}\\
&&&&&&&&&& ((w-2),\vn,(2),\vn^{e-3}) & ((w-2),(1),(1),\vn^{e-3}) && \\ \hline 
$\bullet$ &&& & $\bullet$ && && &&((w-2),(2),\vn^{e-2}) & ((w-2,2),\vn^{e-1}) &  \text{(\ref{targetmatrix})} &\text{\cref{subsec:wt2secondcase,subsec:wt2fourthcase}} \\ 
&&&&&&&&&&((w-2,1),(1),\vn^{e-2}) & ((w-2),(1^2),\vn^{e-3}) &&  \\ \hline 
\end{array}\]
\captionof{table}{A case-by-case analysis of \cref{thm:mostblocks}} \label{tab1}
\end{center}
\end{landscape}
    \clearpage
}
\end{proof}

\begin{eg}
Let $w=e=5$ and $p\geq 0$.
Take $\rho=(10,6,4,3,2^2,1^4)$ and consider the block $B(\rho,5)$.
The abacus display of $\rho$ is given by
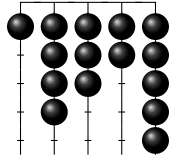
\[\abacus(lmmmr,bbbbb,nbbbb,nbbnb,nbnnb,nnnnb)\]
and we see that 
$p_0=0$, $p_1=8$, $p_2=12$, $p_3=16$ and $p_4=24$ so that
$e<p_{e-1} - p_{e-2} <2e$, $p_{e-2}-p_{e-3}<e$ and $p_{e-1}-p_{e-3}>2e$
and we are in the second case in Table \ref{tab1}.
We therefore take the partitions given by the following four abacus displays.
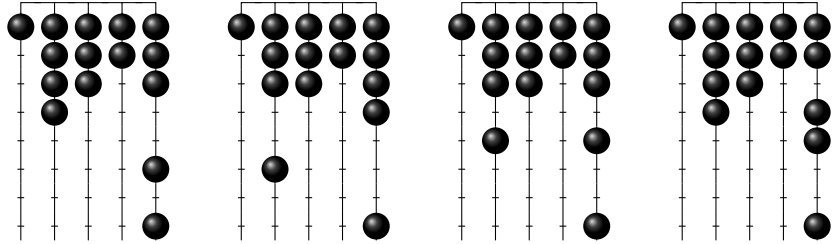
\[
\abacus(lmmmr,bbbbb,nbbbb,nbbnb,nbnnn,nnnnn,nnnnb,nnnnn,nnnnb) \qquad 
\abacus(lmmmr,bbbbb,nbbbb,nbbnb,nnnnb,nnnnn,nbnnn,nnnnn,nnnnb) \qquad 
\abacus(lmmmr,bbbbb,nbbbb,nbbnb,nnnnn,nbnnb,nnnnn,nnnnn,nnnnb) \qquad 
\abacus(lmmmr,bbbbb,nbbbb,nbbnn,nbnnb,nnnnb,nnnnn,nnnnn,nnnnb)
\]
To compute the partial decomposition matrix indexed by these partitions, we first apply row removal, as described in \cref{thm:rowremFock,thm:rowremDecomp}, to remove the first bead from each abacus configuration.
We are then in a block of weight 2 and we follow \cref{subsec:wt2firstcase} to show that the partial decomposition matrix indexed by these partitions is (\ref{targetmatrix}) in characteristic $0$, and that the corresponding ungraded decomposition numbers are characteristic-free, so that applying \cref{prop:matrixtrick} yields that the block $B(\rho,5)$ is Schurian-infinite.
%
\end{eg}

\subsection{When $e=3$ and $p=2$}
We now consider the remaining cases.
For the remainder of this section, fix $e=3$ and $p=2$.

If a core $\rho$ has an abacus display with $s_i$ beads on runner $i$ for $i=0,1,2$ then we will write it as a tuple $[s_0,s_1,s_2]$. By placing $p_0$ in the top left position, we can ensure that there is just one bead on the leftmost runner, so $s_0=1$ while the middle and right-hand runners have $s_1\geq 1$ and $s_2 \geq 1$ beads on, respectively.
By applying the Scopes equivalences of \cref{prop:grScopes}, we may further assume that $s_1 \leq w$ and $s_2 \leq s_1 + w - 1$.
In this way, the triples $[1,s_1,s_2]$ index Scopes classes of blocks of a fixed weight, and we thus identify the triples with the corresponding Scopes classes.
We will sometimes write $B([1,s_1,s_2],w)$ for the associated block, rather than $B(\rho,w)$.

Note that we may also make the following observation to simplify the task of proving that each $B(\rho,w)$ is Schurian-infinite.
If we replace the core $\rho$ with its conjugate, $\rho'$, then $B(\rho,w)$ is isomorphic to $B(\rho',w)$, with an isomorphism being induced by the $\#$-automorphism on $\hhh$ (e.g.~see \cite[Exercise~3.14]{mathas}).
In such a situation, we will say that two Scopes classes are conjugate.

\begin{thm}\label{thm:mostblocks32}
Let $w\geq 4$, and let $\rho$ be a $3$-core.
If $\rho$ satisfies $p_{2} - p_{1} <2e$, the weight $w$ block $B(\rho,w)$ of $\hhh$ is Schurian-infinite.
\end{thm}

\begin{proof}
The proof is almost identical to that of \cref{thm:mostblocks}.
The third and fifth cases in Table~\ref{tab1} do not appear here, since we have assumed that $p_{2} - p_{1} < 2e$.
For all other cases, we may take exactly the same partitions as in that proof, and obtain the result in the same way.
\end{proof}

\begin{rem}
When the weight is $4$, the above theorem covers all Scopes classes except for the following nine:
\begin{alignat*}{3}
&[ 1, 1, 3 ], \qquad &&[ 1, 1, 4 ], \qquad &&[ 1, 2, 4 ]\\
&[ 1, 2, 5 ], &&[ 1, 3, 5 ], &&[ 1, 3, 6 ]\\
&[ 1, 4, 1 ], &&[ 1, 4, 6 ], &&[ 1, 4, 7 ].
\end{alignat*}
If we tried to argue by the same method, our row-removal to reduce to weight 2 would leave us computing in the weight 2 Scopes classes $[1,2,3]$ and $[1,1,2]$.
For the latter, no submatrix of the decomposition matrix will do what we need, but in \cref{subsec:wt2thirdcase} we argued directly by looking at extensions between simples.
One could instead argue that this Scopes class is Schurian-infinite since the conjugate Scopes class $[1,2,2]$ is.
For the former -- the Rouquier block -- the class is self-conjugate, and the extensions do not suffice, so we had to find other means of directly proving that the block is Schurian-infinite.
\end{rem}

We next deal with some special cases when $w=4$.
The Rouquier block for weight $4$ is the block with Scopes class $[1,4,7]$ (c.f.~\cite{leclercmiyachi,jlm}).

\begin{prop}\label{prop:Rouqblocks}
Let $B(s,4)$ be the Rouquier block of weight $4$.
We define the following four partitions in $B(s,4)$ by their $3$-quotients.

\begin{align*}
\la^{(1)} &= ((1^{2}),(1^2),\varnothing), 
&\la^{(2)} &= ((1),(2,1),\varnothing),\\
\la^{(3)} &= ((1),(1^3),\varnothing), &
\la^{(4)} &= (\varnothing,(2,1^2),\varnothing).
\end{align*}
Then the partial decomposition matrix corresponding to these four partitions in both characteristic $0$ and characteristic $2$ is equal to (\ref{targetmatrixalt})
and hence $B(s,4)$ is Schurian-infinite.
\end{prop}

\begin{proof}

It may easily be checked, using either the LLT algorithm or the explicit formula for decomposition numbers for Rouquier blocks in characterstic 0 \cite[Corollary~10]{leclercmiyachi}, that the partial decomposition matrix is as described in characteristic $0$.
Recall \cite[Theorem~5.17]{bk09} that there exists a lower unitriangular matrix $A =(a_{\la\mu})$ - the graded adjustment matrix - with entries in $\mathbb{N}[v]$ and rows and columns indexed by the $e$-regular partitions such that
\begin{equation}
d^{e,p}_{\la\mu}(v) = d^{e,0}_{\la\mu}(v) + \sum_{\nu \triangleleft \mu} d^{e,0}_{\la\nu}(v)a_{\nu\mu}(v).\label{eqn:adj} 
\end{equation} 
Suppose $\nu,\mu \in B(\rho,4)$ have respective $3$-quotients $(\nu^{(0)},\nu^{(1)},\nu^{(2)})$ and $(\mu^{(0)},\mu^{(1)},\mu^{(2)})$.
Since $B(s,4)$ is a Rouquier block, by \cite[Proposition~4.4]{jlm} we have  $a_{\nu\mu}(v) = 0$ unless $|\nu^{(i)}|=|\mu^{(i)}|$ for $i=0,1,2$.
Putting these results together, we see that $d_{\la^{(i)}\la^{(j)}}^{3,2}(v) = d_{\la^{(i)}\la^{(j)}}^{3,0}(v)$ unless $j=2$, with
\[
d^{3,2}_{\la^{(i)} \la^{(2)}}(v) = d^{3,0}_{\la^{(i)}\la^{(2)}}(v) + d^{3,0}_{\la^{(i)}\la^{(3)}} (v) a_{\la^{(3)} \la^{(2)}}(v),
\]
so that the partial decomposition matrices agree if $a_{\la^{(3)}\la^{(2)}}=0$.
If $d^{3,2}_{\la^{(3)}\la^{(2)}}=0$ then certainly $a_{\la^{(3)} \la^{(2)}}=0$, so we find this decomposition number.
Note that there does not exist $\sigma \in B(s,4)$ with $\la^{(3)} \triangleleft \sigma \triangleleft \la^{(2)}$ so since the Jantzen coefficient $J_{\la^{(3)}\la^{(2)}}$ is equal to $0$ we also have $d^{3,2}_{\la^{(3)}\la^{(2)}}=0$.
(For more information on the Janzten sum formula, see \cite[Section~5.2]{mathas}.)
\end{proof}

\begin{rem}
The above argument extends readily to proving that any Rouquier block for $e\geq 3$, $w\geq 4$ and $p \neq 3$ is Schurian-infinite (and there exist other partitions that will prove the case $p \neq 2$).
However, it is slightly cleaner, notationally, to handle this one case alone, and our other methods -- used to prove \cref{thm:mostblocks,thm:wlargemain} -- apply to a much broader collection of blocks.
\end{rem}

The following result holds for any parameters $e$ and $p$, although we only require it for our choice of $e=3$ and $p=2$.

\begin{lem} \label{lem:rest}
Let $\la,\mu \in B(\rho,w)$ be such that $\mu$ is $e$-regular and both $\la$ and $\mu$ have exactly $k$ removable $i$-nodes.
Let $\bar{\la}$ and $\bar{\mu}$ be the partitions obtained by removing $k$ nodes from $\la$ and $\mu$ respectively.
Suppose that $\bar{\la}$ and $\bar{\mu}$ both have exactly $k$ addable $i$-nodes.
Then $\bar{\mu}$ is $e$-regular and $d^{e,p}_{\la\mu}(1) \geq d^{e,p}_{\bar{\la}\bar{\mu}}(1)$.
\end{lem}

\begin{proof}
Apply the functor $i$-Res as described in \cite[Section~6.1]{mathas} to the projective indecomposable module $P(\mu)$.
\end{proof}

\begin{prop}\label{prop:NearRouqblocks}
Suppose that $s \in \{[1,4,6], [1,3,6], [1,3,5]\}$.
We define the following four partitions in $B(s,4)$ by their $3$-quotients.
\begin{align*}
\la^{(1)} &= ((1^{2}),(1^2),\varnothing),
&\la^{(2)} &= ((1),(2,1),\varnothing),\\
\la^{(3)} &= ((1),(1^3),\varnothing),
&\la^{(4)} &= (\varnothing,(2,1^2),\varnothing).
\end{align*}
Then the partial decomposition matrix corresponding to these four partitions in both characteristic $0$ and characteristic $2$ is equal to (\ref{targetmatrixalt})
and hence $B(s,4)$ is Schurian-infinite.
\end{prop}

\begin{proof}
In fact, we prove the stronger claim that the proposition holds for all $s \in \{[1,4,7], [1,7,4],$ $[1,4,6], [4,1,7], [1,3,6], [1,6,3], [1,3,5]\}$.
We may easily verify using the LLT algorithm that the seven partial decomposition matrices are as stated in characteristic $0$, thus by Equation \ref{eqn:adj} we also have lower bounds on the entries in the matrices in characteristic $2$.
The case that $s=[1,4,7]$ is exactly \cref{prop:Rouqblocks}.
We first apply \cref{lem:rest} to obtain upper bounds on the entries for $[1,7,4]$, coming from those known decomposition numbers for the Rouquier block $[1,4,7]$, as above.
Since these agree with the lower bounds, the proposition holds for $[1,7,4]$; an identical argument shows it is true for $[4,1,7]$.
We apply Scopes equivalence to $[1,7,4]$ and $[4,1,7]$ respectively to show that it is true for $[7,1,4]=[1,4,6]$ and $[4,7,1]=[1,3,6]$.
We repeat the argument applying \cref{lem:rest} to $[1,3,6]$ to show the result holds for $[1,6,3]$; finally Scopes equivalence proves that it holds for $[6,1,3]=[1,3,5]$.
\end{proof}

\begin{thm}\label{thm:w4main}
Suppose $e=3$, $p=2$, $w = 4$, and let $s$ be a Scopes class for $w$.
Then the block $B(s,w)$ of $\hhh$ is Schurian-infinite.
\end{thm}

\begin{proof}
All but nine Scopes classes are handled by \cref{thm:mostblocks32}, as discussed in the remark below it.
The Rouquier block $[1,4,7]$ is handled in \cref{prop:Rouqblocks}, while $[1,3,5]$, $[1,3,6]$, and $[1,4,6]$ are dealt with in \cref{prop:NearRouqblocks}.
The remaining Scopes classes $[1,1,3]$, $[1,1,4]$, $[1,2,4]$, $[1,2,5]$, and $[1,4,1]$ are conjugate to the classes $[1,3,3]$, $[1,4,4]$, $[1,3,4]$, $[1,4,5]$, and $[1,4,3]$, respectively, and are thus Morita equivalent to blocks we've already shown to be Schurian-infinite.
\end{proof}

\begin{thm}\label{thm:wlargemain}
Suppose $e=3$, $p=2$, $w \geq 5$, and let $s$ be a Scopes class for $w$.
Then the block $B(s,w)$ of $\hhh$ is Schurian-infinite.
\end{thm}

\begin{proof}
First, suppose the Scopes class is $[1,s_1,s_2]$, with $s_1 \leq s_2$.
Note that $[1,s_1,s_2]' = [1,s_2-s_1+1,s_2]$.
Since $B(s,w)$ and $B(s',w)$ are Morita equivalent, it is sufficient to consider the cases where $s_2-s_1 \leq s_1-1$ so we assume further that this inequality holds.

If $s_2-s_1 \leq 1$, the result follows immediately by \cref{thm:mostblocks32}.

So suppose $s_2-s_1 = 2$, so that $s_1 \geq 3$.
Then we define partitions
\begin{align*}
\la^{(1)} &= ((w-3,2),(1),\varnothing), 
&\la^{(2)} &= ((w-3),(3),\varnothing),\\
\la^{(3)} &= ((w-3,1),(2),\varnothing), 
&\la^{(4)} &= ((w-3,1^2),(1),\varnothing).
\end{align*}
By row-removal \cref{thm:rowremFock,thm:rowremDecomp}, the relevant submatrix of the decomposition matrix matches that of the four partitions obtained by removing the first row from each, or in other words removing the lowest bead from the abacus display of each.
The remaining partitions are then in the block $B([1,s_1,s_1+1], 3)$, which is Scopes equivalent to $B([1,3,4], 3)$ for any $s_1 \geq 3$.
The remaining partitions are precisely those used in \cref{subsec:wt3fifthcase}, and the result follows, by \cref{prop:matrixtrick}.

Next, suppose $s_2-s_1=3$, so that we also have $s_1 \geq 4$.
Then we define partitions
\begin{align*}
\la^{(1)} &= ((w-4,1^2),(1^2),\varnothing), 
&\la^{(2)} &= ((w-4,1),(2,1),\varnothing),\\
\la^{(3)} &= ((w-4,1),(1^3),\varnothing), 
&\la^{(4)} &= ((w-4),(2,1^2),\varnothing).
\end{align*}
Arguing as before, removing the first row yields partitions in the block $B([1,s_1,s_1+2], 4)$, which is Scopes equivalent to $B([1,4,6], 4)$ for any $s_1 \geq 4$.
The remaining partitions are precisely those used in \cref{prop:NearRouqblocks}, and the result follows.

Now we suppose that $s_2-s_1\geq 4$, so that we also have $s_1 \geq 5$.
Then, taking the exact same partitions as in the previous case, and performing row removal as before now yields partitions in the block $B([1,s_1,s_2-1], 4)$, which is Scopes equivalent to $B([1,4,7], 4)$.
Since the remaining partitions are still those used in \cref{prop:Rouqblocks}, the result follows once more.
This completes the proof for all Scopes classes $[1,s_1,s_2]$ for $s_2 \geq s_1$.

We now assume that $s_2<s_1$.
Note that $[1,s_1,s_2]' = [1,s_1,s_1-s_2]$, so it suffices to assume that $s_2 \geq s_1-s_2$.

If $s_1-s_2 \leq 2$, the result follows immediately by \cref{thm:mostblocks32}.

Next, suppose $s_1-s_2=3$, so that $s_2 \geq 3$.
Then we define partitions
\begin{align*}
\la^{(1)} &= ((w-3,2),(1),\varnothing), 
&\la^{(2)} &= ((w-3),(3),\varnothing),\\
\la^{(3)} &= ((w-3,1),(2),\varnothing), 
&\la^{(4)} &= ((w-3,1^2),(1),\varnothing).
\end{align*}
Arguing as before, removing the first row yields partitions in the block $B([1,s_1-1,s_2], 3)$, which is Scopes equivalent to $B([1,3,4], 3)$.
The remaining partitions are precisely those used in \cref{subsec:wt3fifthcase}, and the result follows.

Finally, we suppose that $s_1-s_2\geq4$ so that $s_2 \geq 4$.
Then we define partitions
\begin{align*}
\la^{(1)} &= ((w-4,1^2),(1^2),\varnothing), 
&\la^{(2)} &= ((w-4,1),(2,1),\varnothing),\\
\la^{(3)} &= ((w-4,1),(1^3),\varnothing), 
&\la^{(4)} &= ((w-4),(2,1^2),\varnothing).
\end{align*}

Arguing as before, removing the first row yields partitions in the block $B([1,s_1-1,s_2], 4)$, which is Scopes equivalent to $B([1,s_1-s_2+3,4], 4)$.
In turn, this is Scopes equivalent to $B([1,4,6], 4)$ if $s_1 - s_2=4$ or $B([1,4,7], 4)$ if $s_1 - s_2\geq5$.
The remaining partitions are precisely those used in both of \cref{prop:NearRouqblocks,prop:Rouqblocks}, and the result follows.
\end{proof}

Combining \cref{thm:mostblocks,thm:w4main,thm:wlargemain} yields our main result, \cref{thm:main}, that for $e\neq2$, every block of $\hhh$ of weight at least 2 is Schurian-infinite, and these blocks are thus Schurian-finite if and only if they have finite representation type.

\bibliographystyle{amsalpha}  
\addcontentsline{toc}{section}{\refname}
\bibliography{master}

\end{document}